\documentclass[USenglish]{article}
\usepackage{fullpage}
\usepackage{amsmath,amsfonts,amsthm,amssymb}
\usepackage{color,xcolor}
\usepackage{ifpdf}
\usepackage{psfrag}
\usepackage{graphicx,graphics}

\usepackage[small]{caption}
\usepackage{subcaption}

\usepackage{xparse}
\usepackage{bigints}

\usepackage{url}
\usepackage{hyperref}
\usepackage{todonotes}
\usepackage{ulem} 
\newtheorem{theorem}{Theorem}[section]

\newtheorem{lemma}[theorem]{Lemma}
\newtheorem{proposition}[theorem]{Proposition}
\newtheorem{remark}[theorem]{Remark}
\newtheorem{definition}[theorem]{Definition}
\newtheorem{ex}[theorem]{Example}

\usepackage{cancel}
\usepackage{algorithm,algorithmic}

\newcommand{\R}{\mathbb R}

\newcommand{\N}{\mathbb N}

\newcommand{\E}{\mathrm{E}}

\newcommand{\bs}{\boldsymbol}

\newcommand{\mean}[1]{\overline{#1}}

\newcommand{\Ol}{\mathcal{O}}

\NewDocumentCommand{\mat}{mo}{%
  \IfValueTF{#2}{%
    \underline{\underline{#1}}{#2}
  }{%
    \underline{\underline{#1}}\,
  }%
}
\def\L{\mathcal{L}}
\def\I{\mathcal{I}}
\def\R{\mathbb{R}}
\def\bbc{\underline{\mathbf{c}}}
\def\bc{\mathbf{c}}

\def\bbd{\underline{\mathbf{d}}}
\def\M{\mathrm{M}}
\usepackage{bbm}

\definecolor{darkspringgreen}{rgb}{0., 0.55, 0.3}
\definecolor{dartmouthgreen}{rgb}{0.05, 0.5, 0.06}
\definecolor{etonblue}{rgb}{0.59, 0.78, 0.64}
\definecolor{airforceblue}{rgb}{0., 0.4, 0.66}
\definecolor{arylideyellow}{rgb}{0.91, 0.84, 0.42}
\definecolor{emerald}{rgb}{0.31, 0.78, 0.47}
\definecolor{uclagold}{rgb}{1.0, 0.7, 0.0}
\definecolor{cadmiumorange}{rgb}{0.93, 0.53, 0.18}

\hypersetup{
pdftitle={}
pdfauthor={P. \"Offner and D. Torlo},
pdfpagemode=UseOutlines,
linkbordercolor=0 0 0,
linkcolor=red,
citecolor=blue,
colorlinks=true,
bookmarks = true
}
\begin{document}
\title{Arbitrary high-order, conservative and positive preserving Patankar-type deferred correction schemes }

\author{Philipp \"Offner\thanks{Corresponding author: P. \"Offner, philipp.oeffner@math.uzh.ch}\,  and Davide Torlo \thanks{Corresponding author: D. Torlo, davide.torlo@math.uzh.ch} 
 \\
Institute of Mathematics,
University of Zurich, Switzerland}
\date{August 6th, 2019}
\maketitle

\begin{abstract}

 Production-destruction systems (PDS)  of ordinary differential equations (ODEs)
are used to describe physical and biological reactions in nature. The considered quantities are subject to natural laws.
Therefore, they preserve positivity and conservation of mass at the analytical level.

In order to maintain these properties at the discrete level, the so-called modified Patankar-Runge-Kutta (MPRK) schemes are often used in this context.
However, up to our knowledge, 
the family of MPRK has been only developed up to third order of accuracy. In this work, we propose a method to solve PDS problems, but using the Deferred Correction (DeC) process as a time integration method. Applying the modified Patankar approach to the DeC scheme results in provable conservative and positivity preserving methods.
Furthermore, we demonstrate that these modified Patankar DeC schemes can be constructed up to arbitrarily high order. 
Finally, we validate our theoretical analysis through numerical simulations. 


\end{abstract}

\section{Introduction}\label{sec:intro}

The modelling of geobiochemical processes or ecosystems leads often to systems of ordinary differential equations (ODEs) which 
can be formulated in the so-called \textbf{production-destruction} systems (PDS) as  described in \cite{burchard2005application,hense2010representation} 
for example.
To guarantee the physical and chemical laws, the quantities have to fulfil several conditions like positivity and conservation. 

The applied numerical method should not violate these conditions and big efforts have been devoted to designing conservative 
and positivity preserving schemes, since classical approaches like Runge Kutta (RK) schemes do not guarantee these properties.

In \cite{burchard2003high}
the authors suggest modified Patankar-type methods of first and second order which verify the desired properties, i.e., conservation and positivity.
Recently, further extensions were done to construct modified Patankar-Runge-Kutta (MPRK)
schemes of second and third order \cite{kopecz2018order,kopecz2019existence,kopecz2018unconditionally,huang2018third,huang2019positivity}.
As the name suggests, all these schemes use, as a basic procedure, the Runge-Kutta method, which has been modified by weighting the 
production and destruction terms as suggested in \cite{patankar1980numerical}. Thanks to these weighting coefficients, the schemes are forced to maintain positivity of the variables and to conserve some quantities of interest. However, the described and constructed schemes are, up to our knowledge, at most third order accurate.

In this paper, we present a way to construct \textbf{arbitrary high-order, positivity preserving, numerically robust and conservative} schemes
for PDS. Differently from previous schemes, we do not start building our schemes on RK methods. We consider the Deferred Correction (DeC) procedure, a high order time integration technique, and we modify it, in order to obtain a positivity preserving, conservative and arbitrary high-order scheme. Moreover, we provide a proof of the desired properties.

The paper is organised as follows.\\
In section \ref{sec:PDS} we introduce the production-destruction systems and we give a short introduction about the so-called Patankar trick 
and how it was applied in \cite{burchard2005application} to construct a modified Patankar-type scheme starting from the explicit Euler method.
Afterwards, in section \ref{sec:originalDeC}, we introduce the Deferred Correction (DeC) method and we discuss conservation and positivity for this classical formulation. 
In section \ref{sec:mPDeC}, we build the main core of this work, explaining our modification of DeC through the Patankar trick (mPDeC) and we prove that the obtained mPDeC schemes are positive 
preserving, conservative and arbitrary high-order accurate.
In section \ref{sec:Num}, we validate our theoretical investigations, considering three different benchmark problems, which are also discussed in different literature references, as \cite{burchard2003high,kopecz2018order}.
Finally, we give a summary and an outlook for possible extensions.

\section{Production--Destruction Systems}\label{sec:PDS}

In this paper we consider production-destruction systems (PDS) of the form 
\begin{equation}\label{eq:original_model}
\begin{cases}
d_t c_i = P_i(\bc ) - D_i(\bc ) , \quad i=1,\dots,I,\\
\bc(t=0)=\bc_0,
\end{cases}
\end{equation}
where $\bc = (c_1, \dots , c_I) ^T \in \mathbb{R}^I$ represents the vectors of $I$ constituents, $t$ denotes the time and $\bc_0$ the initial condition.
Moreover, $P_i(\bc)$ and $D_i(\bc)$ represent the production and destruction rates of the $i$-th constituent
and both terms are assumed to be non-negative, i.e, $P_i,\;D_i\geq0$ for $i=1,\dots,I$. 
These systems rise naturally to describe geochemical processes as it is described in
\cite{burchard2003high,burchard2005application}
and we recapitulate their notations and definitions in this section. \\
The production and destruction terms can also be written in a \textit{matrix form} as follows
\begin{equation}\label{eq:matrix}
P_i(\bc) = \sum_{j=1}^I p_{i,j}(\bc), \quad D_i(\bc) = \sum_{j=1}^I d_{i,j}(\bc),
\end{equation} 
where each term $p_{i,j}\geq 0$ and $d_{i,j}\geq 0$ are Lipschitz continuous functions and may depend linearly or non--linearly on $\bc$. 
Furthermore, the term $d_{i,j}$ describes the rate of change from the $i$-th to the $j$-th constituent while $p_{i,j}$ is the rate 
at which the $j$-th constituent is transformed into the $i$-th.\\
We are interested in (fully) conservative and positive production--destruction systems. 
To clarify these expressions we repeat the definitions from \cite{kopecz2018order}.
\begin{definition}\label{def:PDS_positve}
 The PDS \eqref{eq:original_model} is called \textbf{positive} if positive initial values $c_i(0)>0$ for $i=1,\dots, I$ imply positive solutions, $c_i(t)>0$ for $i=1, \cdots, I$ for all times $t>0$.\\
 The PDS \eqref{eq:original_model} is called \textbf{conservative} if at any time $t\geq 0$, we have that 
\begin{equation}\label{eq:conservation}
\sum_{i=1}^I c_i(t) = \sum_{i=1}^I c_i(0)
\end{equation}
is fulfilled. 
In the analytic form \eqref{eq:original_model}, the conservation property \eqref{eq:conservation} is equivalent to the following relation for the matrix representation \eqref{eq:matrix}
\begin{equation}\label{eq:conservation_matrix}
p_{i,j}(\bc)=d_{j,i}(\bc), \quad \forall i,j=1,\dots, I. 
\end{equation}
Moreover, the system is called \textbf{fully conservative}
if additionally $p_{i,i}(\bc)=d_{i,i}(\bc)=0$ holds for all $\bc\geq 0$ and $i=1,\dots, I$. 
\end{definition}

As it is described  in \cite{kopecz2018order} every conservation PDS can be written in a fully conservative formulation.
We can rewrite the two terms of \eqref{eq:conservation_matrix} into one matrix of exchanging quantities $e(\bc)$ defined as
\begin{equation}\label{eq_matrix_zero}
e_{i,j}(\bc) := p_{i,j}(\bc) -d_{i,j}(\bc).
\end{equation} 
Clearly, from property \eqref{eq:conservation_matrix}, we have that $e_{i,i}=0$. With this notation, let us define the total exchange rate for the $i$-th constituent as
\begin{equation}
E_i(\bc):=P_i(\bc)-D_i(\bc).
\end{equation}


%

A numerical method suited to solve a conservative and positive PDS \eqref{eq:original_model} should mimic, 
at the discrete level, the continuous setting properties. For a one-step methods, we can introduce the discrete analogues of definitions \eqref{def:PDS_positve}.
\begin{definition}
 Let $\bc^{n}$ denote the approximation of $\bc(t^{n})$ at the time level $t^n$.  A one-step method
\begin{equation}\label{eq:one_step}
 \bc^{n+1}=\bc^n+\Delta t \Phi(t^n,\bc^n, \bc^{n+1}, \Delta t),
\end{equation}
with process function $\Phi$,  is called 
\begin{itemize}
 \item \textbf{unconditionally conservative} if for all $n\in \N$ and $\Delta t>0$
 \begin{equation}
  \sum_{i=1}^I c^{n+1}_i = \sum_{i=1}^I  c_i^n
 \end{equation}
holds;
\item \textbf{unconditionally positive} if for all $\Delta t >0$ and $\bc^n>0$, we have that 
$\bc^{n+1}>0$.
\end{itemize}
\end{definition}
\begin{ex}\label{Ex:Euler_}
Let us consider as an example the explicit Euler method. The method is defined by
\begin{equation}
\bc^{n+1}= \bc^n +\Delta t E_i(\bc^n).
\end{equation} 
It is conservative 
since 
\begin{equation}
 \sum_{i=1}^I \left( c_i^{n+1}-c_i^n\right)=  \sum_{i=1}^I \left( c_i^{n}+\Delta t \sum_{i=1}^I \left(p_{i,j}(\bc^n)-d_{i,j}(\bc^n)
 \right) -c_i^n\right)=\Delta t \sum_{i=1}^I \left(p_{i,j}(\bc^n)-d_{i,j}(\bc^n)
 \right) =0
\end{equation}
holds. Conversely, the explicit Euler method is not unconditionally positive.
Consider a conservative and positive PDS \eqref{eq:original_model} where we assume that 
the right hand side is not identical zero. Then, there exists 
a $\bc^n\geq0$ such that $\bf{P}(\bc^n)-\bf{D}(\bc^n)\neq0$. Since the PDS is conservative, we can at least
find one constituent $i\in \lbrace  1,\dots, I \rbrace$, where $D_i(\bc^n)>P_i(\bc^n)\geq0$. Choosing
\begin{equation}
\Delta t >\frac{c_i^n}{D_i(\bc^n)-P_i(\bc^n)} > 0,
\end{equation}
we obtain 
\begin{equation}
 c_i^{n+1}=c_i^{n} +\Delta t\left(P_i(\bc^n)-D_i(\bc^n)\right)<c_i^{n} +\frac{\bc_i^n}{D_i(\bc^n)-P_i(\bc^n)} \left(P_i(\bc^n)-D_i(\bc^n)\right)
 =c_i^{n}-c_i^{n}=0.
\end{equation}
This demonstrates the violation of the positivity for the explicit Euler method for unbounded timesteps $\Delta t$. 
\end{ex}
To build an unconditionally positive numerical scheme,
Patankar had the idea in \cite{patankar1980numerical}
of weighting the destruction term in the original explicit Euler methods
with the following coefficient
\begin{equation}\label{eq:patankar}
c_i^{n+1}=c_i^n+\Delta t  \left( \sum_{j=1}^I p_{i,j}(\bc^n) - 
\sum_{j=1}^I d_{i,j}(\bc^n) \frac{c^{n+1}_i}{c_i^n} \right), \quad i=1,\dots, I.
\end{equation}
Hence, the scheme \eqref{eq:patankar} is unconditionally positive, but the conservation relation 
is violated. 
In \cite{burchard2003high} a modification of the Patankar scheme \eqref{eq:patankar} was presented, resulting in an unconditionally positive and conservative method. It is defined as follows.
\begin{equation}\label{eq:mod_patankar}
c_i^{n+1}:=c_i^n+\Delta t  \left( \sum_{j=1}^I p_{i,j}(\bc^n) \frac{c^{n+1}_j}{c_j^n} - \sum_{j=1}^I d_{i,j}(\bc^n) \frac{c^{n+1}_i}{c_i^n} \right), \quad i=1,\dots, I.
\end{equation}
The scheme is implicit and can be solved inverting 
the mass matrix $\M$ in the system $\M\bc^{n+1}=\bc^n$ where $\M$ is
\begin{equation}\label{eq:mass_matrix_mod_patankar}
\begin{cases}
m_{i,i}(\bc^n) = 1+\Delta t \sum_{k=1}^I \frac{d_{i,k}(\bc^n)}{c_i^n} , \quad i=1,\dots, I,\\
m_{i,j}(\bc^n) = - \Delta t \frac{p_{i,j}(\bc^n)}{c_j^n} , \quad i,j=1,\dots , I ,\, i\neq j.
\end{cases}
\end{equation}
The construction of the mass matrix $\M$ must follow substantial prescriptions in order to preserve
the positivity of the scheme, as suggested in \cite{kopecz2018unconditionally}.
\begin{remark}
 Extensions of the modified Patankar scheme \eqref{eq:mod_patankar}
to Runge-Kutta schemes were proposed in
\cite{kopecz2018order,kopecz2018unconditionally} and further developed in \cite{huang2019positivity,huang2018third}.
Special focus lies in the weighting of the production and destruction terms as it 
is investigated for example in \cite{kopecz2019existence} and references therein.
Families of second and third order modified Patankar-Runge-Kutta (MPRK) schemes can 
be found in the mentioned literature.
We do not provide the definition of MPRK because the modified Patankar scheme 
 \eqref{eq:mod_patankar}  already gives us the basic idea for the new methods we want to propose. We will prove that these methods are positivity preserving, conservative and 
arbitrary high-order. \\
\end{remark}

\section{Deferred Correction Methods}\label{sec:originalDeC}

There are various approaches to solve numerically an ODE. 
A first approach is given by finite differences, where the derivative in time is replaced by differences of states in different timesteps. Backward and forward Euler are examples of this kind of strategy. Another approach would be to  reformulate the ODE
by integrating it in time. With different quadrature formulas and approximation techniques 
one can obtain various Runge-Kutta methods (explicit and implicit ones), see \cite{hairer1991solving,wanner1996solving} for details.  However,  we
follow a different approach in this paper.\\
We start our investigation with the \textbf{Deferred Correction (DeC)} method introduced
in \cite{dutt2000dec}.  In its original formulation, it is an explicit, arbitrary high order method
for ODEs. Further extensions of DeC can be found in the literature, including semi-implicit approaches
as in \cite{minion2003dec}. However, in this work we will not consider the semi-implicit framework. Instead, we will focus on the \textbf{explicit} DeC approach used by Abgrall in \cite{abgrall2017dec}.
In our opinion, his notation describes DeC in a more compact way than in 
previous works \cite{dutt2000dec,christlieb2010integral,liu2008strong}.\footnote{We like to mention that Abgrall focused on DeC as a time integration scheme in the context of 
finite element methods. 
Applying a classical RK method, a dense mass matrix has to be inverted and Abgrall 
wanted to avoid this.
 By using a DeC scheme, instead, he showed that a mass matrix free approach is possible \cite{abgrall2017dec}. } 
Nevertheless, the main idea is always the same 
and it is based on the Picard-Lindel\"of theorem in the continuous setting. 
The theorem states the existence and uniqueness of 
solutions for ODEs. The classical proof makes use of the so-called Picard iterations to minimize the error and to prove convergence. The foundation of DeC relies on mimicking the Picard iterations at the
discrete level. The approximation error decreases with several iteration steps. 
For the description of DeC, Abgrall
introduces two operators: $\L^1$ and $\L^2$.\\
Here, 
the $\L^1$ operator represents a low-order easy-to-solve numerical scheme, e.g. the explicit Euler method, 
and $\L^2$ is a high-order operator that can present difficulties in its practical solution, e.g. an implicit RK scheme.
The DeC method can be written as a combination of these two operators.\\
Given a timeinterval $[t^n, t^{n+1}]$ we subdivide  it into $M$ subintervals  $\lbrace [t^{n,m-1},t^{n,m}]\rbrace_{m=1}^M$, where $t^{n,0} = t^n$ and $t^{n,M} = t^{n+1}$ and we mimic for every 
subinterval $[t^0, t^m]$ the Picard--Lindel\"of theorem for both operators $\L^1$ and $\L^2$. We drop the dependency on the timestep $n$ for subtimesteps $t^{n,m}$ and substates $\bc^{m,n}$ as denoted in Figure \ref{Fig:Time_interval}.  
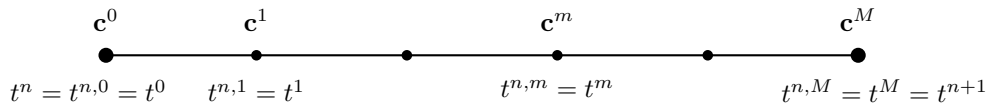
\begin{figure}[ht]
	\centering
\begin{tikzpicture}
\draw [thick]   (0,0) -- (10,0) node [right=2mm]{};
\fill[black]    (0,0) circle (1mm) node[below=2mm] {$t^n=t^{n,0}=t^0 \,\, \quad$} node[above=2mm] {$\bc^0$}
                            (2,0) circle (0.7mm) node[below=2mm] {$t^{n,1}=t^1$} node[above=2mm] {$\bc^1$}
                            (4,0) circle (0.7mm) node[below=2mm] {}
                            (6,0) circle (0.7mm) node[below=2mm] {$t^{n,m}=t^m$} node[above=2mm] {$\bc^m$}
                            (8,0) circle (0.7mm) node[below=2mm] {}
                            (10,0) circle (1mm) node[below=2mm] {$\qquad t^{n,M}=t^M=t^{n+1}$} node[above=2mm] {$\bc^M$}; 
\end{tikzpicture} \caption{Figure: divided time interval}\label{Fig:Time_interval}
\end{figure}

Then, the $\L^2$ operator is given by
\begin{equation}\label{eq:L2operator}
\L^2(\bc^0, \dots, \bc^M) :=
\begin{cases}
\bc^M-\bc^0 -\int_{t^0}^{t^M} \I_M ( E(\bc^0),\dots,E(\bc^M))\\
\vdots\\
\bc^1-\bc^0 - \int_{t^0}^{t^1} \I_M ( E(\bc^0),\dots,E(\bc^M))
\end{cases}.
\end{equation}
Here, the term $\I_M$  denotes an interpolation polynomial of order $M$ evaluated at the points $\lbrace  t^{n,r}\rbrace _{r=0}^M$. In particular, we use Lagrange polynomials $\lbrace \varphi_r \rbrace_{r=0}^M$, where $\varphi_r(t^{n,m})=\delta_{r,m}$ and $\sum_{r=0}^M \varphi_r(s) \equiv 1$ for any $s\in [0,1]$. 
Using these properties, 
we can actually compute the integral of the interpolants, thanks to a quadrature rule in the same points $\lbrace t^m \rbrace_{m=0}^M$ with weights $\theta_r^m := \int_{t^n}^{t^{n,m}} \varphi_r(s) ds$. 
We can rewrite
\begin{equation}\label{eq:L2}
\L^2(\bc^0, \dots, \bc^M) =
\begin{cases}
\bc^M-\bc^0 - \sum_{r=0}^M \theta_r^M E(\bc^r)\\
\vdots\\
\bc^1-\bc^0 - \sum_{r=0}^M \theta_r^1 E(\bc^r)
\end{cases}.
\end{equation}
The $\L^2$ operator represents an $(M+1)$ order numerical scheme if set equal
to zero, i.e., $\L^2(\bc^0, \dots, \bc^M)=0$. Unfortunately, the resulting scheme is implicit and, further, the terms $E$ may be non-linear. Because of this, the only $\L^2$ formulation is not explicit and more efforts have to be  made to solve it.

For this purpose, we introduce a simplification of  the $\L^2$ operator. Instead of using a quadrature formula at the points $\lbrace t^m \rbrace_{m=0}^M$ we evaluate 
the integral in equation \eqref{eq:L2operator} applying the left Riemann sum. 
The resulting operator $\L^1$ is given by the forward Euler discretization for each state $\bc^m$ in the timeinterval, i.e.,  
\begin{equation}\label{eq:L1}
\L^1(\bc^0, \dots, \bc^M) :=
\begin{cases}
 \bc^M-\bc^0 - \beta^M \Delta t E(\bc^0) \\
\vdots\\
\bc^1- \bc^0 - \beta^1 \Delta t E(\bc^0)
\end{cases}.
\end{equation}
with coefficients $\beta^m:=\frac{t^m-t^0}{t^M-t^0}$.\\
To simplify the notation and to describe  DeC, we  introduce the vector of states for the variable $\bc$ at all subtimesteps\footnote{We provide a table with all definitions and notations in the appendix \ref{sec:notation}.}
\begin{align}\label{eq:definition_bbc}
&\bbc :=  (\bc^0, \dots, \bc ^M) \in \R^{M\times I}, \text{ such that }\\
&\L^1(\bbc) := \L^1(\bc^0, \dots, \bc^M) \text{ and } \L^2(\bbc) := \L^2(\bc^0, \dots, \bc^M) .
\end{align}
Now, the DeC algorithm uses a combination of the $\L^1$ and $\L^2$ operators
to provide an iterative procedure. The aim is to recursively approximate $\bbc^*$, the numerical solution of
the $\L^2=0$ scheme, similarly to the Picard iterations in the 
continuous setting. The successive states of the iteration process will be denoted 
by the superscript $(k)$, where $k$ is the iteration index, e.g. $\bbc^{(k)}\in \R^{M\times I}$.
The total number of iterations (also called correction steps in the following) is denoted by $K$.
To describe the procedure, we have to refer to both 
 the $m$-th subtimestep and the $k$-th iteration of the DeC algorithm. We will indicate the variable by $\bc^{m,(k)} \in \R^I$.
 Finally, the DeC method can be written as \\

 \centerline{\textbf{DeC Algorithm}}
%
%
\begin{equation}\label{DeC_method}
\begin{split}
&\bc^{0,(k)}:=\bc(t^n), \quad k=0,\dots, K,\\
&\bc^{m,(0)}:=\bc(t^n),\quad m=1,\dots, M\\
&\L^1(\bbc^{(k)})=\L^1(\bbc^{(k-1)})-\L^2(\bbc^{(k-1)}) \text{ with }k=1,\dots,K,
\end{split}
\end{equation}
where $K$ is the number of iterations that we  want to compute. 
Using the procedure \eqref{DeC_method}, we  need, in particular, as many iterations as the desired  order of accuracy, i.e., $K=d= M+1$. 

Notice that, in every step, we solve the equations for the unknown variables $\bbc^{(k)}$ which appears only in the $\L^1$ formulation, the operator that can be easily inverted. Conversely, $\L^2$ is only applied to already computed predictions of the solution $\bbc^{(k-1)}$.
Therefore, the scheme \ref{DeC_method} is completely explicit and arbitrary high order as stated in \cite{abgrall2017dec} with
 the following proposition.
\begin{proposition}\label{DeC_prop}
Let $\L^1$ and $\L^2$ be two operators defined on $\mathbb{R}^M$, which depend on the discretization scale $\Delta = \Delta t$, such that
\begin{itemize}
\item $\L^1$ is coercive with respect to a norm, i.e., $\exists\, \alpha_1 >0$ independent of $\Delta$, such that for any $\bbc,\bbd$ we have that $$\alpha_1||\bbc-\bbd||\leq ||\L^1 (\bbc)-\L^1 (\bbd)||,$$
\item $\L^1 - \L^2$ is Lipschitz with constant $\alpha_2>0$ uniformly with respect to $\Delta$, i.e., for any $\bbc,\bbd$
$$
||(\L^1(\bbc)-\L^2(\bbc))-(\L^1(\bbd)-\L^2(\bbd))||\leq \alpha_2 \Delta ||\bbc-\bbd||.
$$
\end{itemize}
We also assume that there exists a unique $\bbc^*_\Delta$ such that $\L^2(\bbc^*_\Delta)=0$. Then, if $\eta:=\frac{\alpha_2}{\alpha_1}\Delta<1$, the DeC is converging to $\bbc^*$ and after $k$ iterations the error $||\bbc^{(k)}-\bbc^*||$ is smaller than $\eta^k||\bbc^{(0)}-\bbc^{*}||$.

\end{proposition}

\begin{remark}
The DeC procedure is naturally conservative if $\L^1$ is conservative, but it is not positivity preserving if $\L^1$ is positivity preserving. Indeed, the coefficients $\theta_r^m$ of the operator $\L^2$ can be negative and spoil the positivity of the scheme. This is one of the points that make us modify the classical DeC into the scheme that we propose in this work.
\end{remark}

\begin{remark}
Any DeC scheme can be interpreted as a RK scheme \cite{christlieb2010integral}. The main difference between RK and  DeC is that the latter gives a general approach to the time discretization and does not require a specification of the coefficients for every order of accuracy. 
On the other side, by rewriting a DeC method as a RK scheme, it
requires a number of stages equal to $K\times M=d\times(d-1)$, which is bigger than classical RK stages. However,  one can notice that every subtimestep is independent of another, so one can compute sequentially the corrections and in parallel the subtimesteps, obtaining a computational cost of just $K=d$ corrections.

\end{remark}

\begin{ex}\label{example}
For clarity, we provide here an example of a second order DeC scheme. To get this order of accuracy, we need $K=2$ DeC iterations and one subtimestep $[t^n=t^{n,0}, t^{n,1}=t^{n+1}]$.
 Reminding that $\bc^{0,(k)}=\bc(t^n)\, \forall k$, the method \eqref{DeC_method} for the first step reads

 \begin{equation*}
 \begin{split}
  &\L^1(\bbc^{(1)})\stackrel{!}{=} \L^1(\bbc^{(0)})-\L^{2}(\bbc^{(0)})\\
 \Longleftrightarrow  &c^{1,(1)}_i-c^{0,(0)}_i -\Delta t \beta^1 E_{i,j}( \bc^{0,(1)})=\\
 &c^{1,(0)}_i-c^{0,(0)}_i -\Delta t \beta^1 E_{i}( \bc^{0,(0)}) \\
 -&c^{1,(0)}_i+c^{0,(0)}_i +\Delta t \sum_{r=0}^M \theta_r^1  E_{i}( \bc^{r,(0)})\\
 \Longleftrightarrow & c^{1,(1)}_i= c^{0,(0)}_i +\Delta t E_{i}( \bc^{0,(0)}) = c^{0,(0)}_i +\Delta t \sum_{j=1}^I \left( p_{i,j}( \bc^{0,(0)}) - d_{i,j}( \bc^{0,(0)}) \right)
 \end{split}
 \end{equation*}
Substituting this term into the first correction steps leads finally to
\begin{equation*}
 \begin{split}
  &\L^1(\bc^{(2)})=\L^1(\bc^{(1)})-\L^2(\bc^{(1)})\\
  \Longleftrightarrow&
  c^{1,(2)}_i-c^{0,(2)}_i -\Delta t E_{i}( \bc^{0,(2)})\\
 &= c^{1,(1)}_i-c^{0,(1)}_i -\Delta t E_{i}( \bc^{0,(1)})\\ 
 &-c^{1,(1)}_i+c^{0,(1)}_i+\sum_{r=0}^1 \theta_r^1 \Delta t 
 E_{i}( \bc^{r,(1)})
 \end{split}
\end{equation*}
The correction step is not modifying the initial subtimestep. Therefore,  with $\bc^{0,(1)}=\bc^{0,(2)}$, we get
\begin{equation*}
 c_i^{n+1}=c^{1,(2)}_i=c^{0,(0)}_i+ \sum_{r=0}^1 \theta_r^1 \Delta t \sum_{j=1}^I\left(
  p_{i,j}( \bc^{r,(1)}) - d_{i,j}( \bc^{r,(1)}) \right)
\end{equation*}
where $\theta_0^1=\theta_1^1=\frac{1}{2}$. 
This scheme  coincides with the strong stability preserving Runge-Kutta method of 
second order \cite{gottlieb2011strong}. 

\end{ex}

\begin{remark}
Before we  modify our DeC framework, we want to give some final remarks.\\ 
The presented  DeC approach is not the most general version. 
In our description we always include both endpoints in  the point distribution of the subtimesteps, i.e., $t^0=t^n$ and $t^M=t^{n+1}$.
However, this is not necessary, as it is already described in \cite{dutt2000dec}, where 
also Gauss-Legendre nodes are applied. Then, the approximation at the endpoint 
is done via extrapolation. Nevertheless, we do not consider in this work this class of point distribution.\\
Secondly, instead of using the explicit Euler method in  $\L^1$, explicit high-order RK methods can also be applied. In principle, this yields a faster increase of the order of accuracy in the iterative procedure, but it has been shown, that it leads also to some problems of smoothness of the error behaviour as it is described  in \cite{christlieb2010integral}, which results in a drop down of the expected accuracy order. 
However, we will consider this approach in future research. 
\end{remark}

\section{Modified Patankar Deferred Correction Scheme}\label{sec:mPDeC}

In this section, we are going to propose a positivity preserving, conservative and  
arbitrary high-order scheme, 
that will be denoted as  \textbf{modified Patankar Deferred Correction (mPDeC)}. \\
The DeC procedure \eqref{DeC_method}
serves us as a starting point to construct this scheme,
and, thanks to its structure, we will be able to prove the hypotheses of 
Proposition \ref{DeC_prop}.
This yield us directly the desired order condition for our modified DeC scheme 
without performing a specific Taylor expansion for every order of accuracy.
We will adapt DeC in such a way to obtain all the properties we are interested in.\\
The conservation can be easily guaranteed by the consistency of the two operators,
i.e., by the consistency of the two schemes described by  $\L^1$ and by $\L^2$. \\
Conversely, more effort is required to produce a positivity preserving scheme. For this purpose, we follow 
the ideas of Patankar \cite{patankar1980numerical} and Burchard et al. 
\cite{burchard2003high} of weighting the destruction and production terms 
in the scheme. Their aim is to obtain a mass matrix shaped as in the 
modified Patankar scheme \eqref{eq:mod_patankar}
where all of the positive terms are collected on the diagonal, while the negative 
terms are put in the non-diagonal entries. 
This will guarantee that the mass matrix is diagonally dominant by columns, with positive 
diagonal values, and, thus, its inverse will be positive. Therefore, we introduce some 
coefficients similar to the ones proposed in \eqref{eq:mod_patankar}.\\
Finally, as we have seen in the example \eqref{Ex:Euler_}, an explicit scheme is not positivity preserving
and the investigation in  \cite{burchard2003high,kopecz2019existence,huang2019positivity}
support our decision to modify the DeC scheme in order to get a \emph{fully} implicit 
method. \\
Because of all the above mentioned considerations, we came to the conclusion 
of modifying the $\L^2$ operator, to make it fully implicit. 
In particular, it has to depend on both the previous and the current corrections 
of the DeC procedure. We redefine it as follows.
\begin{equation}\label{eq:L2m}
\begin{split}
&\L^2(\bc^{0,(k-1)}, \dots, \bc^{M,(k-1)},\bc^{0,(k)}, \dots, \bc^{M,(k)}) =\L^2(\bbc^{(k-1)},\bbc^{(k)}):=\\
&\begin{cases}
c_i^{M,(k-1)}-c_i^{0,(k-1)} - \sum\limits_{r=0}^M \theta_r^M \Delta t \sum\limits_{j=1}^I \left(  p_{i,j}(\bc^{r,(k-1)}) \frac{c^{M,(k)}_{\gamma(j,i, \theta_r^M)}}{c_{\gamma(j,i, \theta_r^M)}^{M,(k-1)}} -  d_{i,j}(\bc^{r,(k-1)})  \frac{c^{M,(k)}_{\gamma(i,j, \theta_r^M)}}{c_{\gamma(i,j, \theta_r^M)}^{M,(k-1)}} \right), \forall i=1,\dots, I \\
\vdots\\
c_i^{1,(k-1)}-c_i^{0,(k-1)} - \sum\limits_{r=0}^M \theta_r^1 \Delta t \sum\limits_{j=1}^I
\left( p_{i,j}(\bc^{r,(k-1)}) \frac{c^{1,(k)}_{\gamma(j,i, \theta_r^1)}}{c_{\gamma(j,i, \theta_r^1)}^{1,(k-1)}} -  d_{i,j}(\bc^{r,(k-1)})  \frac{c^{1,(k)}_{\gamma(i,j, \theta_r^1)}}{c_{\gamma(i,j, \theta_r^1)}^{1,(k-1)}} \right), \forall i=1,\dots, I 
\end{cases},
\end{split}
\end{equation}
where $\gamma(a,b, \theta) = a$ if $\theta>0$ and $\gamma(a,b, \theta) = b$ if $\theta <0$.
\begin{remark}\label{L2_new_remark}
The modification of the scheme is done only through the coefficients
$\frac{c^{m,(k)}_{j}}{c^{m,(k-1)}_{j}}$ on both the production and the destruction terms.
The fact that these coefficients depend on the new correction $(k)$ means that we are modifying the
mass matrix of the whole DeC correction step. \\
These coefficients allow to choose in which term of the mass matrix we want to put each term 
$\theta_r^mp_{i,j}$ and $\theta_r^md_{i,j}$, according to the sign of the $\theta$ coefficient. 
The pseudo-algorithm \ref{algo:mass} provides the construction steps of the mass matrix, see  \ref{sec:algorithm}. There, it is straightforward to see that the diagonal terms are all positive and the off--diagonal are all negative.
The index $\gamma$ takes care of the sign of the destruction and production terms which 
are added in the mass matrix. It is inspired by the explanation given in  \cite[Remark 2.5]{kopecz2018order}, that states that,
when negative entries in the Butcher Tableau of the RK scheme appear, one has
to interchange the destruction terms with the production ones to guarantee the positivity preserving property. 
With the $\gamma$ function we are taking this into account.
In our opinion, it is complicated and unclear to investigate higher order $(>3)$ RK schemes properties because of these exchanges depending on the Butcher Tableau. While, with this DeC approach, we can in few lines generalize every order scheme.\\
Moreover, it is helpful to notice that the coefficients that we are using to modify the contributions, namely $\frac{c^{m,(k)}_{j}}{c^{m,(k-1)}_{j}}$, are
 converging to 1 as the iteration index of the DeC increases. 
In subsection \ref{sec:order_DeC} we will make this statement more precise and we will 
study how fast these coefficients converge to 1.
\end{remark}

Most of the terms in the $\L^1$ operator will cancel out through the iteration process,
therefore we keep the $\L^1$  operator as presented 
in the original DeC \eqref{eq:L1}. 

\begin{equation}\label{eq:L1_mPDeC}
\begin{split}
&\L^1(\bc^{0,(k)}, \dots, \bc^{M,(k)}) =\\
&\begin{cases}
c_i^{M,(k)}-c_i^{0,(k)} -  \beta^M\Delta t \left( \sum\limits_{j=1}^I p_{i,j}(\bc^{0,(k)})
 - \sum\limits_{j=1}^I d_{i,j}(\bc^{0,(k)})  \right),
\forall i=1,\dots, I \\
\vdots\\
c_i^{1,(k)}-c_i^{0,(k)} - \beta^1  \Delta t  \left( \sum\limits_{j=1}^I p_{i,j}(\bc^{0,(k)}) 
 - \sum\limits_{j=1}^I d_{i,j}(\bc^{0,(k)})   \right), 
\forall i=1,\dots, I 
\end{cases}.
\end{split}
\end{equation}
Now, we propose the modified Patankar DeC scheme as follows.\\

 \centerline{\textbf{mPDeC Algorithm}}
\begin{equation}\label{DeC_method_new}
\begin{split}
&\bc^{0,(k)}:=\bc(t^n), \quad k=0,\dots, K,\\
&\bc^{m,(0)}:=\bc(t^n),\quad m=1,\dots, M\\
&\L^1(\bbc^{(k)})=\L^1(\bbc^{(k-1)})-\L^2(\bbc^{(k-1)},\bbc^{(k)}) \text{ with }k=1,\dots,K.
\end{split}
\end{equation}

%

One can notice that, using the fact that initial states $c_i^{0,(k)}$
are identical for any correction $(k)$, the DeC correction step \eqref{DeC_method_new}
can be rewritten for $k=1,\dots,K$, $m =1,\dots, M$ and $\forall i\in I$ into
\begin{equation}\label{eq:explicit_dec_correction}
c_i^{m,(k)}-c^0_i -\sum_{r=0}^M \theta_r^m \Delta t  \sum_{j=1}^I 
\left( p_{i,j}(\bc^{r,(k-1)}) 
\frac{c^{m,(k)}_{\gamma(j,i, \theta_r^m)}}{c_{\gamma(j,i, \theta_r^m)}^{m,(k-1)}}
- d_{i,j}(\bc^{r,(k-1)})  \frac{c^{m,(k)}_{\gamma(i,j, \theta_r^m)}}{c_{\gamma(i,j, \theta_r^m)}^{m,(k-1)}} \right)=0.
\end{equation}
We keep both formulations \eqref{DeC_method_new} and \eqref{eq:explicit_dec_correction}
to prove different properties. The DeC formulation \eqref{DeC_method_new}
will help us to demonstrate the accuracy order of the scheme whereas
 formulation \eqref{eq:explicit_dec_correction} will be used to prove conservation and positivity.
Before we start to prove these properties, we give a small example 
 to get used to the formulation \eqref{DeC_method_new}.
 Furthermore, we like to mention that although the algorithm 
 \eqref{DeC_method_new} seems quite complex, it is actually easy to implement. We
 put a small pseudo-code in the  \ref{sec:algorithm} and we refer to the repository \footnote{\url{https://git.math.uzh.ch/abgrall_group/deferred-correction-patankar-scheme}} for a Julia version of the code.

\begin{ex}\label{example_2}
 We give a small example of the constructed method, applying the DeC approach at
 second order of accuracy, as already considered in example \ref{example}, i.e., $K=2$ DeC iterations and  one subtimestep $[t^n=t^{n,0}, t^{n,1}=t^{n+1}]$. 
 In this case, we recall that $\theta_{0}^1=\theta_{1}^1=\frac{1}{2}$ and that $\bc^{0,(0)} = \bc^{1,(0)}$.

 The method \eqref{DeC_method_new} for the first step reads
 \begin{equation*}
 \begin{split}
  & \L^1(\bbc^{(1)})-\L^1(\bbc^{(0)})+\L^2(\bbc^{(0)},\bbc^{(1)})\stackrel{!}{=}0 \\
 \Longleftrightarrow  & c^{1,(1)}_i -c^{0,(1)}_i- \Delta t \sum_{j=1}^I \left( p_{i,j}( \bc^{0,(1)}) - d_{i,j}( \bc^{0,(1)}) \right)= \\
 & c^{1,(0)}_i -c^{0,(0)}_i- \Delta t \sum_{j=1}^I \left( p_{i,j}( \bc^{0,(0)}) - d_{i,j}( \bc^{0,(0)}) \right) \\
 -& c^{1,(0)}+c^{0,(0)}_i +\Delta t \sum_{r=0}^1 \theta^1_r \sum_{j=1}^I \left( p_{i,j}( \bc^{r,(0)})
 \frac{c_j^{1,(1)}}{c_j^{1,(0)}}-  d_{i,j}( \bc^{r,(0)})
 \frac{c_i^{1,(1)}}{c_i^{1,(0)}} \right)\\
 \Longleftrightarrow  & c^{1,(1)}_i =c^{0,(0)}_i +\Delta t \sum_{j=1}^I \left( p_{i,j}( \bc^{0,(0)})
 \frac{c_j^{1,(1)}}{c_j^{1,(0)}}-  d_{i,j}( \bc^{0,(0)})
 \frac{c_i^{1,(1)}}{c_i^{1,(0)}} \right),\\ 
 \end{split}
 \end{equation*}
 where the last step is obtained considering, again the fact that for the iteration $(0)$ all the states coincide. Collecting the mass matrix terms as in \eqref{eq:mass_matrix_mod_patankar}, one can solve the previous equation for $\bc^{1,(1)}$. Substituting this term into the second iteration step leads finally to
\begin{equation*}
 \begin{split}
  &\L^1(\bbc^{(2)})=\L^1(\bbc^{(1)})-\L^2(\bbc^{(1)},\bbc^{(2)})\\
  \Longleftrightarrow&
  c^{1,(2)}_i-c^{0,(2)}_i -\Delta t \sum_{j=1}^I p_{i,j}( \bc^{0,(2)})
 + \sum_{j=1}^I d_{i,j}( \bc^{0,(2)})
 \\
 &= c^{1,(1)}_i-c^{0,(1)}_i -\Delta t \sum_{j=1}^I p_{i,j}( \bc^{0,(1)})
 + \sum_{j=1}^I d_{i,j}( \bc^{0,(1)}) \\ 
 &-c^{1,(1)}_i+c^{0,(1)}_i+\sum_{r=0}^1 \theta_r^1 \Delta t \left(
 \sum_{j=1}^I p_{i,j}( \bc^{r,(1)})
 \frac{c_j^{1,(2)}}{c_j^{1,(1)}} -\sum_{j=1}^I d_{i,j}( \bc^{r,(1)})
 \frac{c_i^{1,(2)}}{c_i^{1,(1)}}\right).
 \end{split}
\end{equation*}
The correction step has no effect on the initial subtimestep. Therefore, we get with $\bc^{0,(1)}=\bc^{0,(2)}$:
\begin{equation*}
 c_i^{n+1}=c^{1,(2)}_i=c^{0,(0)}_i+ \sum_{r=0}^1 \theta_r^1 \Delta t \left(
 \sum_{j=1}^I p_{i,j}( \bc^{r,(1)})
 \frac{c_j^{1,(2)}}{c_j^{1,(1)}} -\sum_{j=1}^I d_{i,j}( \bc^{r,(1)})
 \frac{c_i^{1,(2)}}{c_i^{1,(1)}}\right)
\end{equation*}
where $\theta_0^1=\theta_1^1=\frac{1}{2}$. 
This scheme  coincides with a modified Runge Kutta Patankar scheme of second order as it is presented in \cite{kopecz2018order} .

\end{ex}
\subsection{Conservation and positivity of modified Patankar DeC}
In this section, we are proving that the proposed scheme is unconditionally conservative and positivity preserving.
\begin{theorem}
The  mPDeC scheme in \eqref{eq:explicit_dec_correction} is unconditionally conservative for all substages, i.e., $$\sum_{i=1}^I c^{m,(k)}_i=\sum_{i=1}^I c^{0}_i,$$ for all $k=1,\dots, K$ and $m=0,\dots,M$.
\end{theorem}
\begin{proof}
Using formulation \eqref{eq:explicit_dec_correction}, we can easily see that $\forall k,m$ 
\begin{align}
&\sum_{i\in I} c_i^{m,(k)} - \sum_{i\in I} c_i^{0} = \\ \label{eq:cons_step_def}
=&\Delta t \sum_{i,j=1}^I \sum_{r=0}^M \theta_r^m\left(
 p_{i,j}(\bc^{r,(k-1)}) \frac{c^{m,(k)}_{\gamma(j,i, \theta_r^m)}}{c_{\gamma(j,i, \theta_r^m)}^{m,(k-1)}} - d_{i,j}(\bc^{r,(k-1)})  \frac{c^{m,(k)}_{\gamma(i,j, \theta_r^m)}}{c_{\gamma(i,j, \theta_r^m)}^{m,(k-1)}} 
 \right) =\\ \label{eq:cons_step_prod}
 =&\Delta t \sum_{i,j=1}^I \sum_{r=0}^M\theta_r^m \left(
 d_{j,i}(\bc^{r,(k-1)}) \frac{c^{m,(k)}_{\gamma(j,i, \theta_r^m)}}{c_{\gamma(j,i, \theta_r^m)}^{m,(k-1)}} - d_{i,j}(\bc^{r,(k-1)})  \frac{c^{m,(k)}_{\gamma(i,j, \theta_r^m)}}{c_{\gamma(i,j, \theta_r^m)}^{m,(k-1)}} 
 \right) =\\
 =&\Delta t \sum_{r=0}^M \theta_r^m\left(
 \sum_{i,j=1}^I d_{j,i}(\bc^{r,(k-1)}) \frac{c^{m,(k)}_{\gamma(j,i, \theta_r^m)}}{c_{\gamma(j,i, \theta_r^m)}^{m,(k-1)}} 
 -\sum_{i,j=1}^I d_{i,j}(\bc^{r,(k-1)})  \frac{c^{m,(k)}_{\gamma(i,j, \theta_r^m)}}{c_{\gamma(i,j, \theta_r^m)}^{m,(k-1)}} 
 \right) =0.
\end{align}
To get this result,
we have just used the definition of the scheme 
\eqref{eq:explicit_dec_correction} in \eqref{eq:cons_step_def} and 
the property \eqref{eq:conservation_matrix} of the production and destruction operators $d_{i,j}=p_{j,i}$ 
in \eqref{eq:cons_step_prod}. In the last step, we have exchanged the sums over $j$ and $i$.

\end{proof}

To demonstrate the positivity of the scheme, we introduce some preliminary results.
\begin{lemma}\label{th:diagonal_dominant}
The mass matrix of every correction step of the mPDeC scheme described in \eqref{eq:explicit_dec_correction} is diagonal dominant by columns.
\end{lemma}
\begin{proof}
At each step $(m,k)$ we are solving an implicit linear system where the mass matrix is given by
\begin{equation}\label{eq:mass_matrix}
\M(\bc^{m,(k-1)})_{ij}= \begin{cases}
1+\Delta t \sum\limits_{r=0}^M \sum\limits_{l=1}^I \frac{\theta_r^m}{c_i^{m,(k-1)}} \left( d_{i,l}(\bc^{r,(k-1)})\mathbbm{1}_{\lbrace \theta^m_r>0\rbrace} -p_{i,l} (\bc^{r,(k-1)})\mathbbm{1}_{\lbrace \theta^m_r<0 \rbrace} \right) \quad &\text{for }i=j\\
-\Delta t \sum\limits_{r=0}^M \frac{\theta_r^m}{c_j^{m,(k-1)}} \left( p_{i,j}(\bc^{r,(k-1)})\mathbbm{1}_{\lbrace \theta^m_r>0\rbrace} -d_{i,j} (\bc^{r,(k-1)})\mathbbm{1}_{\lbrace \theta^m_r<0\rbrace} \right) \quad &\text{for }i\ne j
\end{cases}.
\end{equation}
Under the assumption
that $p_{i,j}$ and $d_{i,j}$ are always positive, 
it is straightforward to see that all the terms of the 
sum of $\M(\bc^{m,(k-1)})_{ii}$ are positive by construction and that 
all the terms of the sum of the non-diagonal terms $\M(\bc^{m,(k-1)})_{ij}$ for $i\neq j$ are
negative. Moreover, we can demonstrate that \begin{equation}
|\M(\bc^{m,(k-1)})_{ii}|=\M(\bc^{m,(k-1)})_{ii} > \sum_{j=1, j\ne i}^I- \M(\bc^{m,(k-1)})_{ji}  =\sum_{j=1, j\ne i}^I | \M(\bc^{m,(k-1)})_{ji}|,
\end{equation}
by showing
\begin{equation}
\begin{split}
\M(\bc^{m,(k-1)})_{ii} &= 1+\Delta t \sum\limits_{r=0}^M \sum\limits_{j=1}^{ I} \frac{\theta_r^m}{c_i^{m,(k-1)}} \left( d_{i,j}(\bc^{r,(k-1)})\mathbbm{1}_{\lbrace \theta^m_r>0\rbrace} -p_{i,j} (\bc^{r,(k-1)})\mathbbm{1}_{\lbrace \theta^m_r<0 \rbrace} \right) >\\
&>\Delta t \sum\limits_{r=0}^M \sum\limits_{j=1}^{I} \frac{\theta_r^m}{c_i^{m,(k-1)}} \left( p_{j,i}(\bc^{r,(k-1)})\mathbbm{1}_{\lbrace \theta^m_r>0\rbrace} -d_{j,i} (\bc^{r,(k-1)})\mathbbm{1}_{\lbrace \theta^m_r<0 \rbrace} \right) = \\
&=-\sum\limits_{j=1 ,  j\ne i}^I \M(\bc^{m,(k-1)})_{ji}=\sum_{j=1, j\ne i}^I| \M(\bc^{m,(k-1)})_{ji}|,
\end{split}
\end{equation}
where we have used the property of the $p$ and $d$ matrices 
to obtain the previous computation.  Finally, this proves that the mass matrix is diagonally dominant by columns.
\end{proof}

Using Lemma \ref{th:diagonal_dominant} we prove the following theorem.
\begin{theorem}
The mPDeC scheme defined in \eqref{eq:explicit_dec_correction} is positivity preserving, i.e., if $\bc^0>0$ then $\bc^{m,(k)}>0 $, for all $m=1,\dots, M$ and $k=1,\dots, K$.
\end{theorem}
\begin{proof}
Using lemma \ref{th:diagonal_dominant},
we can prove that the inverse of any mass matrix 
obtained from the DeC iterations is positive, i.e., $(\M^{-1})_{ij}\geq 0,\, \forall i,j$. The proof follows the path of what was proposed in \cite{kopecz2018order}.
Using the Jacobi method, we can converge to $\M^{-1}$ with iterative matrices $Z^{(s)}$ for $s\in \N$, where 
\begin{equation}
Z^{(s+1)}:=(I-D^{-1} \M)Z^{(s)} + D^{-1} \text{, with }Z^{(0)}=I.
\end{equation}
Here, $I$ is the identity and $D$ is the diagonal of $\M$. If we denote the iteration matrix as $B:=I-D^{-1}\M$, we can see that it has spectral radius smaller than one, since $\M$ is diagonally dominant. This means that the Jacobi method is convergent to $\M^{-1}$. Now, since $B>0$ and $D^{-1}>0$ from previous lemma \ref{th:diagonal_dominant} and, by induction, also $Z^{(s)}$ is positive, 
we can say that $\M^{-1}=\lim\limits_{s\to \infty} Z^{(s)}$ will be positive.

\end{proof}

\subsection{Convergence order}\label{sec:order_DeC}
To prove that the solution of the mPDeC procedure is high-order accurate, 
we mimic the proof of the original DeC convergence as in \cite{abgrall2017dec}.
We denote by $\bbc^*$ the solution of the $\L^2$ operator, i.e., $\L^2(\bbc^*, \bbc^*)=0$.
This solution $\bbc^*$ coincides with the solution of the classical $\L^2$ operator 
defined in \eqref{eq:L2operator}. \\
We want to prove that for each iteration step the following 
inequalities are fulfilled: 
%
\begin{align}
||\bbc^{(k)}-\bbc^*||\leq & C_0|| \L^1(\bbc^{(k)}) - \L^1(\bbc^*) ||= \label{eq:coercivity}\\
= & C_0|| \L^1(\bbc^{(k-1)}) -\L^2(\bbc^{(k-1)}, \bbc^{(k)}) - \L^1(\bbc^*) +\L^2(\bbc^*, \bbc^*) ||\leq \label{eq:same_states}\\
\leq & C \Delta t || \bbc^{(k-1)}- \bbc^* || \label{eq:lipschitz}
\end{align}
which implies that for each iteration step we obtain one order of accuracy more than the previous iteration.  
After $K$ iterations we, finally, get 
\begin{equation}
||\bbc^{(K)}-\bbc^*||\leq C^K \Delta t^K ||\bbc^0 -\bbc^*||.
\end{equation}
To prove that  the inequalities \eqref{eq:coercivity} and \eqref{eq:lipschitz} are valid,
we have to demonstrate the following
\begin{enumerate}
\item the coercivity  of the operator $\L^1$ (as in the inequality \eqref{eq:coercivity})
\item the Lipschitz inequality for operator $\L^1 -\L^2$ used in \eqref{eq:lipschitz}
\item the high-order accuracy of the operator $\L^2$, i.e., $||\bbc^* -\bbc^{exact} ||\leq C_d\Delta t ^d$.
\end{enumerate}


Let us start with the coercivity lemma.

\begin{lemma}[Coercivity of $\L^1$]\label{th:coercivity}
Given any $\bbc^{(k)}, \bbc^{*} \in \R^{M\times I}$, there exists a positive $C_0$, such that, the operator $\L^1$ verifies 
\begin{equation}\label{eq:coercivity_teo}
||\L^1(\bbc^{(k)}) - \L^1(\bbc^{*})||\geq C_0||\bbc^{(k)}-\bbc^{(*)}||.
\end{equation}
\end{lemma}
\begin{proof}

We remind  that the beginning states coincide for all the variables, 
i.e., $\bc^{0,(k)}=\bc^{0,*}=\bc^0$.
So, the evolution part simplifies in the two operators and we get the following relation
\begin{align}\label{eq:coercivity_1part}
&\L^1( \bbc^{(k)}) - \L^1(\bbc^{*})= (\bbc^{(k)}-\bbc^{(*)}). 
\end{align} 
This proves that with constant $C_0=1$ the equation \eqref{eq:coercivity_teo} holds.

\end{proof}

Before proving Lipschitz continuity, we need two lemmas. The first one proves that each stage of the scheme is a first order approximation of the previous timestep.
\begin{lemma}\label{th:taylor_1st_order}
For every subtimestep $m=1,\dots,M$ and correction $k=1,\dots, K$, there exists a 
matrix $G$, such that
\begin{equation}
\bc^{m,(k)}=\bc^0 + \Delta t G(\bc^{m,(k-1)}) \bc^0
\end{equation}
holds. Moreover, $G(\bc^{m,(k-1)})=W(\bc^{m,(k-1)})+\mathcal{O}(\Delta t)$, where $W$ does not depend on $\Delta t$. 
\end{lemma}
\begin{proof}
For any $m=0,\dots, M$ and $k=0,\dots, K$, 
the 
equation \eqref{eq:explicit_dec_correction} 
tells us that the mass matrix $\M(\bc^{m,(k-1)})$ can be written as
$\M(\bc^{m,(k-1)})=I-\Delta t W(\bc^{m,(k-1)})$ where $W$ does not 
depend on $\Delta t$, but only on $\bc^{m,(k-1)}$ and the production--destruction functions. It is defined as 
\begin{equation}\label{eq:W_matrix}
W(\bc^{m,(k-1)})_{ij}= \begin{cases}
-\sum\limits_{r=0}^M \sum\limits_{l=1}^I \frac{\theta_r^m}{c_i^{m,(k-1)}} \left( d_{i,l}(\bc^{r,(k-1)})\mathbbm{1}_{\lbrace \theta^m_r>0\rbrace} -p_{i,l} (\bc^{r,(k-1)})\mathbbm{1}_{\lbrace \theta^m_r<0 \rbrace} \right) \quad &\text{for }i=j\\
+\sum\limits_{r=0}^M \frac{\theta_r^m}{c_j^{m,(k-1)}} \left( p_{i,j}(\bc^{r,(k-1)})\mathbbm{1}_{\lbrace \theta^m_r>0\rbrace} -d_{i,j} (\bc^{r,(k-1)})\mathbbm{1}_{\lbrace \theta^m_r<0\rbrace} \right) \quad &\text{for }i\ne j
\end{cases}.
\end{equation}
 This leads to an inverse 

 \begin{equation*}
  (\M(\bc^{m,(k-1)}))^{-1} = I+\Delta t W(\bc^{m,(k-1)})
  + \mathcal{O}(\Delta t^2).
 \end{equation*}
Now, we can define $G$ by
 \begin{equation*}
  G(\bc^{m,(k-1)}):=\frac{1}{\Delta t} \left((\M(\bc^{m,(k-1)}))^{-1} - I \right) =
  W(\bc^{m,(k-1)}) + \mathcal{O}(\Delta t).
 \end{equation*}
So, we can write  
\begin{equation}
\bc^{m,(k)} = (\M(\bc^{m,(k-1)}))^{-1}\bc^0=\bc^0  + \Delta t G(\bc^{m,(k-1)}) \bc^0.
\end{equation}

\end{proof}

With the following lemma, 
we prove that the mPDeC process generates a 
Cauchy sequence similar to the continuous Picard iterations. Moreover, at each iteration, we differ 
from the previous step by an error of one order of accuracy
more. We will drop the
dependency on the subtimestep $m$, as all the relations hold for all of them.
\begin{lemma}\label{lemma:quotient}
 Let $\bc^{(k)}$ and $\bc^{(k-1)} \in \R^I$ verifying Lemma \ref{th:taylor_1st_order},
 then 
 \begin{equation}\label{eq:quotient}
  \frac{c_i^{(k)}}{c_i^{(k-1)}} =1+\Delta t^{k-1}g_i +\Ol(\Delta t^k)
 \end{equation}
holds where $g_i$ are constants independent from $\Delta t$. 
\end{lemma}

\begin{proof}
 We  prove the lemma by induction.\\
 For $k=1$,  equation \eqref{eq:quotient} follows directly from Lemma \ref{th:taylor_1st_order}, i.e., 
 $  \frac{c_i^{(1)}}{c_i^{(0)}}=1+\Ol(\Delta t)$.\\
Given $k\in \N$, as induction hypothesis, \eqref{eq:quotient} holds for $k$, i.e.,
 \begin{equation}\label{eq_induction}
   c_i^{(k)}=c_i^{(k-1)}\left( 1+\Delta t^{k-1} g_i\right)+\Ol(\Delta t^k),
 \end{equation}
where $g_i=G_i(\bbc^{(k-1)}) \bc^0$  and $G_i$ denotes the $i$th row of the matrix $G$.
 We can prove that \eqref{eq:quotient} is verified also for $k+1$.
 Using Lemma \eqref{th:taylor_1st_order}, we obtain 
 \begin{equation*}
 \begin{aligned}
   \frac{c_i^{(k+1)}}{c_i^{(k)}} =&\frac{c_i^{(0)}+\Delta t G_i(\bbc^{(k)})\bc^{(0)}  }{c_i^{(0)}+\Delta t G_i(\bbc^{(k-1)}) \bc^{(0)}  }=\\
   =&\frac{\left(c_i^{(0)}+\Delta t G_i(\bbc^{(k)})\bc^{(0)} \right) \left(c_i^{(0)}-\Delta t G_i(\bbc^{(k-1)}) \bc^{(0)}\right) }{\left(c_i^{(0)}+\Delta t G_i(\bbc^{(k-1)}) \bc^{(0)}\right)\left(c_i^{(0)}-\Delta t G_i(\bbc^{(k-1)}) \bc^{(0)}\right) }=\\
     =&\frac{\left(c_i^{(0)} \right)^2+\Delta t c_i^{(0)} G_i(\bbc^{(k)})\bc^{(0)}  
     -\Delta t c_i^{(0)} G_i(\bbc^{(k-1)}) \bc^{(0)}}{\left(c_i^{(0)}\right)^2- 
 \left(\Delta t G_i(\bbc^{(k-1)}) \bc^{(0)}\right)^2 }  + \\
   &-\frac{\left(\Delta t G_i(\bbc^{(k-1)}) \bc^{(0)} \right)\left(\Delta t G_i(\bbc^{(k)}) \bc^{(0)} \right) } 
   {\left(c_i^{(0)}\right)^2- \left(\Delta t G_i(\bbc^{(k-1)}) \bc^{(0)}\right)^2 }.
 \end{aligned}
 \end{equation*}
 Inserting the induction step \eqref{eq_induction}
 we get
 \begin{equation*}
  \begin{aligned}
    \frac{c_i^{(k+1)}}{c_i^{(k)}} =&\frac{\left(c_i^{(0)}\right)^2+\Delta t c_i^{(0)} \left(G_i\left(\bc^{(k-1)}\bullet \left( \bs{1}+\Delta t^{k-1} {\bf{g}}\right)+\Ol(\Delta t^k) \right) 
   - G_i(\bc^{(k-1)}) \right) \bc^{(0)}  
    }{\left(c_i^{(0)}\right)^2- 
    \left(\Delta t G_i(\bc^{(k-1)}) \bc^{(0)}\right)^2 }+
    \\
    &- \frac{\left(\Delta t G_i(\bc^{(k-1)}) \bc^{(0)} 
    \right)\left(\Delta t G_i\left(\bc^{(k-1)}\bullet \left( \bs{1}+\Delta t^{k-1} {\bf{g}}\right)+\Ol(\Delta t^k) \right) \bc^{(0)} \right) }
   {\left(c_i^{(0)}\right)^2- 
    \left(\Delta t G_i(\bc^{(k-1)}) \bc^{(0)}\right)^2 }
  \end{aligned}
 \end{equation*}
 Here, $\bullet$ denotes the Hadamard product and $\bs{1}:=(1,\dots,1)^T\in \R^I$. 
 The induction step is evaluated for every 
 entry $i$. Using the regularity of $G_i$, we expand its Taylor series in $\bc^{(k-1)}$ for 
 every constituent $i$. Thanks again to the result of Lemma \eqref{th:taylor_1st_order}, we can write
   \begin{equation*}
   \begin{aligned}
    \frac{c_i^{(k+1)}}{c_i^{(k)}} =&\frac{\left(c_i^{(0)}\right)^2+\Delta t c_i^{(0)} G_i\left(\bc^{(k-1)} 
    \right) \bc^{(0)}  + \Delta t^{k} c_i^{(0)} \nabla G_i(\mean{\bc}){\bf{g}} \bc^{(0)}
    - \Delta t c_i^{(0)} G_i(\bc^{(k-1)})  \bc^{(0)} }{
    \left(c_i^{(0)}\right)^2- \left(\Delta t G_i(\bc^{(k-1)}) \bc^{(0)}\right)^2 } +\\
     -&\frac{\left(\Delta t G_i(\bc^{(k-1)}) \bc^{(0)} 
    \right)\left(\Delta t G_i\left(\bc^{(k-1)} \right)\bc^{(0)} + \Delta t^{k}\nabla G_i(\mean{\bc}) {\bf{g}} \bc^{(0)} +\Ol(\Delta t^k)  \right) }{
    \left(c_i^{(0)}\right)^2- \left(\Delta t G_i(\bc^{(k-1)}) \bc^{(0)}\right)^2} 
      \end{aligned}
 \end{equation*}
 where $\mean{\bc}$ is the point of the Lagrange form of the remainder of the Taylor expansion. Hence, we can proceed as follows
   \begin{equation*}
   \begin{aligned}
     \frac{c_i^{(k+1)}}{c_i^{(k)}}=&\frac{\left(c_i^{(0)}\right)^2+\Delta t c_i^{(0)} G_i\left(\bc^{(k-1)} 
    \right) \ \bc^{(0)}  + \Delta t^{k} c_i^{(0)}\nabla G_i(\mean{\bc}) {\bf{g}} \bc^{(0)}
    - \Delta t c_i^{(0)} G_i(\bc^{(k-1)})  \bc^{(0)} }{
    \left(c_i^{(0)}\right)^2- \left(\Delta t G_i(\bc^{(k-1)}) \bc^{(0)}\right)^2 } +\\
     -&\frac{\left(\Delta t G_i(\bc^{(k-1)}) \bc^{(0)} \right)^2 +
    \Ol(\Delta t^{k+1})}{
    \left(c_i^{(0)}\right)^2- \left(\Delta t G_i(\bc^{(k-1)}) \bc^{(0)}\right)^2} =\\
    =&\frac{ \left(c_i^{(0)}\right)^2- \left(\Delta t G_i(\bc^{(k-1)}) \bc^{(0)}\right)^2 +\Delta t^kc_i^{(0)} \nabla G_i(\mean{\bc}) {\bf{g}} \bc^{(0)} +\Ol(\Delta t^{k+1})
    } 
    {\left(c_i^{(0)}\right)^2- \left(\Delta t G_i(\bc^{(k-1)}) \bc^{(0)}\right)^2} = \\
    =&1+\Delta t^k \hat{g}_i +\Ol(\Delta t^{k+1})
  \end{aligned}
  \end{equation*}
which finally proves equation \eqref{eq:quotient} for $k+1$. 
 \end{proof}

Now, let us prove on the Lipschitz continuity of the operator $\L^1-\L^2$.
\begin{lemma}[Lipschitz continuity of $\L^1-\L^2$]\label{th:lipschitz}
Let $\bbc^{(k)}, \bbc^{(k-1)}, \bbc^{*} \in \R^{M\times I}_+$ fulfil Lemma \ref{th:taylor_1st_order}. 
Then, the operator $\L^1 -\L^2$ is Lipschitz continuous with constant $\Delta t C_L$, i.e.,
\begin{equation}
||\L^1 (\bbc^{(k-1)}) -\L^2 (\bbc^{(k-1)}, \bbc^{(k)}) -\L^1 ( \bbc^*) +\L^2 (\bbc^*,\bbc^*)||\leq C_L \Delta t ||\bbc^{(k-1)} -\bbc^*|| .
\end{equation}
\end{lemma}
\begin{proof}
Now, we apply Lemma \eqref{lemma:quotient} to 
substitute the new $\L^2$ operator \eqref{eq:L2m} with 
the original one of the classical DeC \eqref{eq:L2} adding an error of order 
$\Delta t^{k-1}$ to the operator. We get another order from the time integration, 
such that
$$\L^2(\bbc^{(k-1)}, \bbc^{(k)} ) = \L^2(\bbc^{(k-1)} ) + \mathcal{O}(\Delta t^k)$$
and, trivially, $\L^2(\bbc^*, \bbc^*) = \L^2(\bbc^*)$ holds. Together, we obtain 
\begin{align}
&\left\lvert \left \lvert \L^1 (\bbc^{(k-1)}) -\L^2 (\bbc^{(k-1)}, \bbc^{(k)}) -\L^1 ( \bbc^*) +\L^2 (\bbc^*,\bbc^*) \right \rvert \right \rvert \leq \\
\leq & \left\lvert \left \lvert 
\L^1 (\bbc^{(k-1)}) - \L^2 (\bbc^{(k-1)}) -\L^1 (\bbc^*) +\L^2 (\bbc^*)\right 
\rvert \right \rvert + \mathcal{O}(\Delta t^k). \label{eq:Second_term}
\end{align}
Now, we have to take care
about the different variables in the operators. Let us start studying the operator $\L^1 -\L^2$.
We are focusing on each line of the schemes for an arbitrary subtimestep. The difference is given by
\begin{equation}
\begin{split}
&\L^{1,m}_{i}(\bc^{(k-1)})-\L^{2,m}_i(\bc^{(k-1)}) =\\
& \bigintsss_{t^0}^{t^m} \I_M\left(\left\lbrace E_{i}(\bc^{r,(k-1)})   \right\rbrace_{r=0}^M \right) - \I_0\left(\left\lbrace  E_{i}(\bc^{r,(k-1)})   \right\rbrace_{r=0}^M \right) dt\\
& =\bigintsss_{t^0}^{t^m} (\I_M-\I_0)\left(\left\lbrace E_{i}(\bc^{r,(k-1)})  \right\rbrace_{r=0}^M \right)dt.
\end{split}
\end{equation}
Now, we can compute the difference of the two terms
\begin{equation}
\begin{split}
&||\L^{1,m}_{i}(\bc^{(k-1)})-\L^{2,m}_i(\bc^{(k-1)})-\L^{1,m}_{i}(\bc^*)+\L^{2,m}_i(\bc^*)|| =\\
=&\left\vert\left\vert \bigintsss_{t^0}^{t^m} (\I_M-\I_0)\left(\left\lbrace  E_{i}(\bc^{r,(k-1)})-E_{i}(\bc^{r,*})  \right\rbrace_{r=0}^M \right)dt \right \vert\right \vert\leq\\
\leq & \Delta t C_1 ||E_i(\bbc^{(k-1)})-E_i(\bbc^*) || \leq \\
\leq & \Delta t C_L ||\bc^{(k-1)}-\bc^* ||.
\end{split}
\end{equation}
In last step, we have used the regularity of the solutions $\bbc^{(k-1)}$ and $\bbc^{*}$ and the fact that $\I_M -\I_0$ brings 
an error of order zero $\mathcal{O}(1)$ times $\Delta t$ given
by the time integration. Then, we have used the Lipschitz continuity of
the functions $E_{i}$.\\
Overall, the constant $C_L$ depends on the operators $p$ and $d$
 and the Lemma is proven. 

\end{proof}

Finally, we need to show that the solution $\bc^*$ of the operator $\L^2(\bc^*, \bc^*)=0$ is an $(M+1)$-order accurate solution. This is given directly by the definition of the operator \eqref{eq:L2m}, since it is an $(M+1)$-order accurate approximation of the original problem \eqref{eq:original_model} when the two input coincide and, thus, the modification coefficients become 1 and the operator becomes the original one \eqref{eq:L2}.

\begin{theorem}[Convergence of mPDeC]\label{th:final_teo}

Let $\L^1(\cdot)$ and $\L^2(\cdot, \cdot)$ be the operators defined in \eqref{eq:L1_mPDeC} and \eqref{eq:L2m} respectively. The mPDeC procedure \eqref{DeC_method_new} gives an approximation solution with order of accuracy equal to $\min(M+1,K)$.
\end{theorem}

\begin{proof}
With Lemma \ref{th:coercivity} we proved the coercivity of the operator $\L^1$, which verifies the inequality in \eqref{eq:coercivity}. The definition of the mPDeC scheme \eqref{DeC_method_new} gives us the equality 
\eqref{eq:same_states} and the Lipschitz continuity lemma \ref{th:lipschitz} proves the inequality \eqref{eq:lipschitz}. 
Moreover, we know that $\bbc^*$ is an $(M+1)$-accurate approximation of the $\bbc^{ex}$ exact solution.

So, overall, we have
\begin{equation}
||\bbc^{(K)}-\bbc^{ex}||\leq ||\bbc^{*}-\bbc^{ex}|| + (C\Delta t)^K ||\bbc^{*}-\bbc^{(0)}||\leq C^* \Delta t^{M+1} + (C\Delta t)^K.
\end{equation}

\end{proof}

All the desired properties (unconditionally positivity, unconditionally 
conservation and high-order accuracy) are fulfilled by the proposed scheme.

\section{Numerics}\label{sec:Num}

In this section, we validate our theoretical investigation of section \ref{sec:mPDeC} considering some test cases from \cite{kopecz2018order,burchard2003high}.
We focus here only on systems of ordinary differential equations (ODE) (stiff and non-stiff). 
However, the mPDeC schemes can be in general used as time-integration methods for a semidiscrete formulation of partial differential equations, where the spatial discretization is already provided by RD, DG, FR,
(c.f. \cite{abgrall2018connection,abgrall2018asymptotic,ranocha2016summation}) or your favourite space discretization method.

As part of future research, we will consider these schemes in real applications like non-equilibrium flows or shallow water equations as it was already done, for example,
for MPRK together with a WENO approach in \cite{huang2019positivity} or a DG one in \cite{meister2014unconditionally}. In this work we focus on systems of ODEs.
In all the numerical tests, we applied the mPDeC approach on equidistant subtimestep points distributions.
\subsection{Linear Model}
We start by considering a simple linear test case proposed in \cite{burchard2003high,meister2014unconditionally}.
The initial value problem for the PDS is given by 
\begin{equation}\label{eq:linear_test}
\begin{aligned}
 &c_1'(t)=c_2(t)-5c_1(t),\quad  &c_2'(t)=5c_1(t)-c_2(t),\\
 &c_1(0)=c_1^0=0.9, \quad  &c_2(0)=c_2^0=0.1 \, .
\end{aligned}
\end{equation}
The initial values of \eqref{eq:linear_test} are positive and we can rewrite the right hand side of the ODE system in a PDS format 
as follows 
\begin{equation*}
  p_{1,2}(\bc)=d_{2,1}(\bc)=c_2, \quad p_{2,1}(\bc)=d_{1,2}(\bc)=5c_1 
\end{equation*}
and $p_{i,i}(\bc)=d_{i,i}(\bc)=0$ for $i=1,2$. The system describes the exchange of mass between two constituents. 
The analytical solution is given by
\begin{equation}\label{eq:analytical}
 c_1(t)=\frac{1}{6}\left(1 +  \frac{13}{5} \exp(-6t) \right) \text{ and } c_2(t)=1-c_1(t).
\end{equation}

\begin{figure}[!htp]
\centering
    \includegraphics[width=0.8\textwidth]{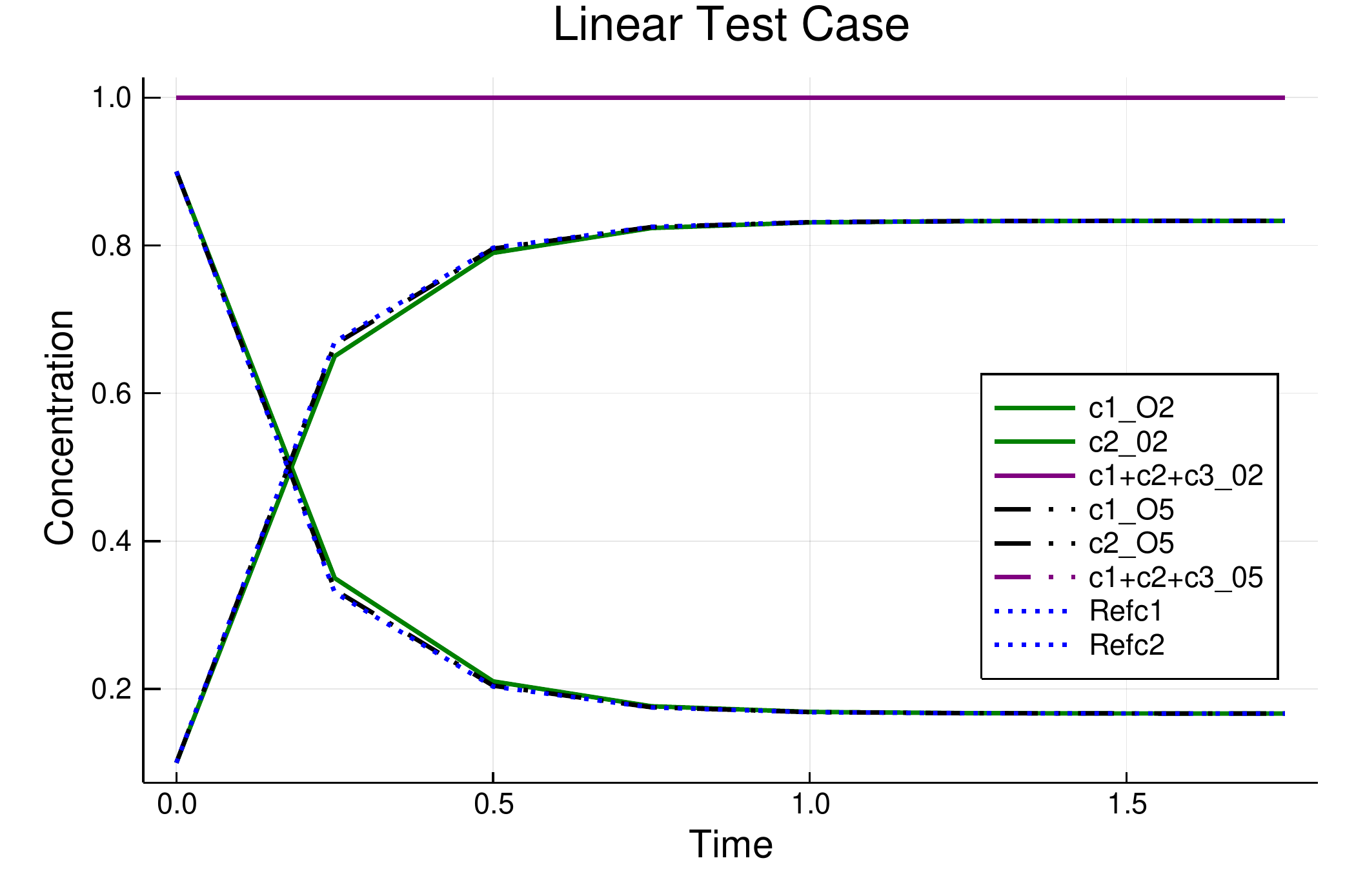}
  \caption{ Second and fifth order methods together with the reference solution \eqref{eq:analytical}
 }
  \label{fig:Linear_model}
\end{figure}

The problem is considered on the time interval  $[0,1.75]$ and,  analogously to \cite{burchard2003high}, 
we use $\Delta t=0.25$
in the simulations. In Figure \eqref{fig:Linear_model} we plot the analytical solution (dotted, blue line) 
and the approximated solutions using 2-nd (solid line, green) and 5-th (dash-dotted line, black) order 
mPDeC methods.
The purple lines represent the sum of the constituents and they are constantly equal to 1, since the methods are conservative. 
Qualitatively, we see that the 5-th order method approximates better the analytical solutions. Furthermore, to verify the order of convergence of the methods, we 
consider also the error behavior of the different order schemes. 
Differently from Kopecz and Meister \cite{kopecz2018order,kopecz2018unconditionally}
instead of calculating the relative errors, we compute the absolute discrete $L^2$ errors
taken over all the timesteps $\lbrace t^n \rbrace_{n=0}^N$ and all the constituents:
\begin{equation}\label{eq:error_formula}
 \E=\frac{1}{N} \sum_{n=1}^N \left( \frac{1}{I} \sum_{i=1}^I \left(c_i(t^n)-c_i^n \right)^2 \right)^\frac{1}{2}.
\end{equation}
After a comparison between the final time error and the one proposed \eqref{eq:error_formula}, we do not observe much discrepancy. Therefore, we will provide only results obtained with \eqref{eq:error_formula}.

In Figure \ref{fig:Linear_model_error}, the left picture shows the error decay for mPDeC schemes at different discretization scales $\Delta t$. In the right picture, we plot the slope of the error decay for different orders of accuracy. 
These graphs demonstrate the high-order accuracy of the proposed methods and the expected convergence rates, validating the theoretical results.
It is also possible to test the scheme with higher order of accuracy.
However, we have notice a reduction of the order as we reach orders higher than $10$, 
probably due to Runge phenomena.
These are well known issues that arise also with the usual DeC methods
\cite{dutt2000dec} using equidistant points distribution in the subtimesteps. 
A possible solution of this problem can be the usage of Gau\ss-Lobatto nodes as point distributions.
This and stability investigations will be part of future research.
\begin{figure}[!htp]
\centering
    \includegraphics[width=0.48\textwidth]{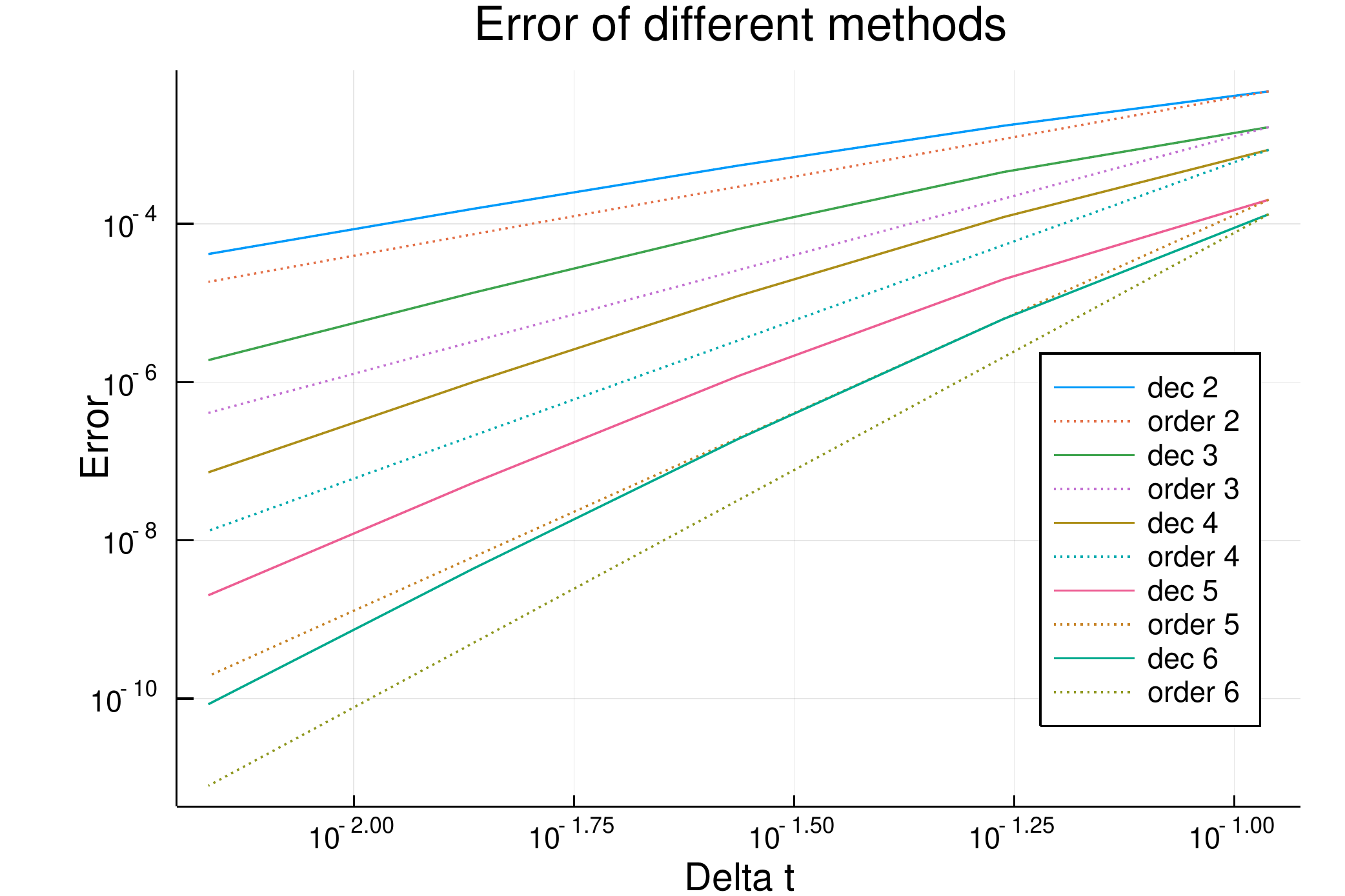}
    \includegraphics[width=0.48\textwidth]{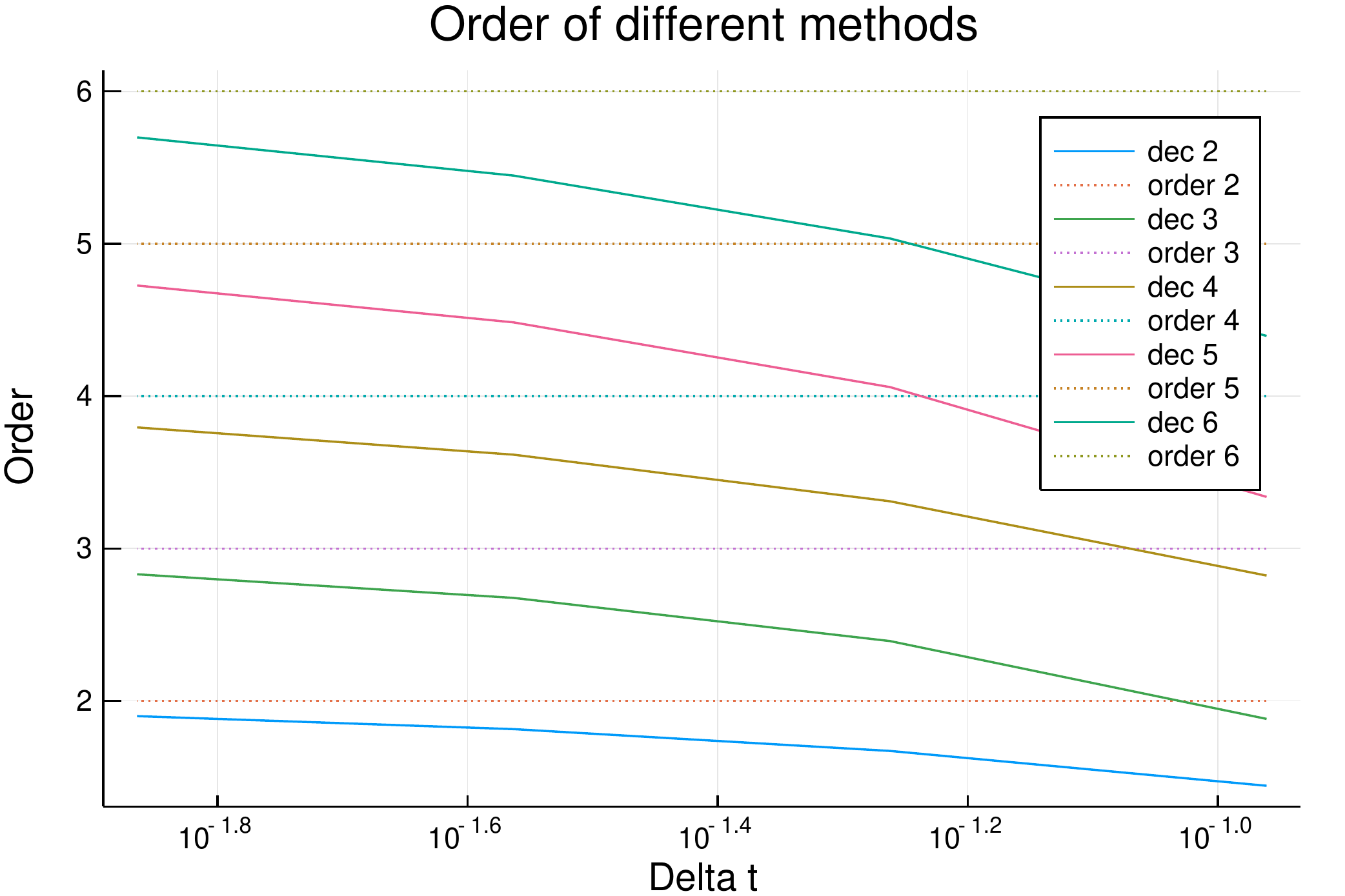}
  \caption{ Second to sixth order error decay and slope of the errors
 }
  \label{fig:Linear_model_error}
\end{figure}

\subsection{Nonlinear test problem}
In this next subsection, we consider the nonlinear test problem 
\begin{equation}\label{eq:nonlinear_test}
 \begin{aligned}
  c_1'(t)&=-\frac{c_1(t)c_2(t)}{c_1(t)+1},\\
  c_2'(t)&= \frac{c_1(t)c_2(t)}{c_1(t)+1}-0.3c_2(t),\\
  c_3'(t)&=0.3c_2(t)
 \end{aligned}
\end{equation}
with initial condition $\bc^0=(9.98,0.01,0.01)^T$. As before, this problem was proposed in \cite{kopecz2018order}. The PDS system in the matrix formulation can be expressed by 
\begin{equation*}
 p_{2,1}(\bc)=d_{1,2}(\bc)=\frac{c_1(t)c_2(t)}{c_1(t)+1}, \quad  p_{3,2}(\bc)=d_{2,3}(\bc)=0.3c_2(t)
\end{equation*}
and $p_{i,j}(\bc)=d_{i,j}(\bc)=0$ for all other combinations of $i$ and $j$. 
This system \eqref{eq:nonlinear_test} is used to describe an algal bloom, that transforms nutrients $c_1$ via phytoplankton $c_2$ into detritus $c_3$.
In our test, we consider the time interval $[0,30]$ and $\Delta t=0.5$. 
We calculate the reference solution with the 
strong stability preserving Runge-Kutta method 10 stages 4th order introduced by Ketcheson 
\cite{ketcheson2008highly}, further investigated in \cite{ranocha2018L2stability} and
implemented in Julia, see \cite{rackauckas2017differentialequations}
for details.  \\
In Figure \ref{fig:Non_linear}, the 6-th order mPDeC (black, dash-dotted lines) approximates very precisely the reference solution.
The 2-nd order method (solid line, green) shows the same structure as the reference solution but it exhibits a severe error. However, 
the approximated second order solution is comparable with the results obtained in \cite{kopecz2018order}.
We see again that the conservation property is fulfilled in the purple lines. 
\begin{figure}[!htp]
\centering
    \includegraphics[width=0.8\textwidth]{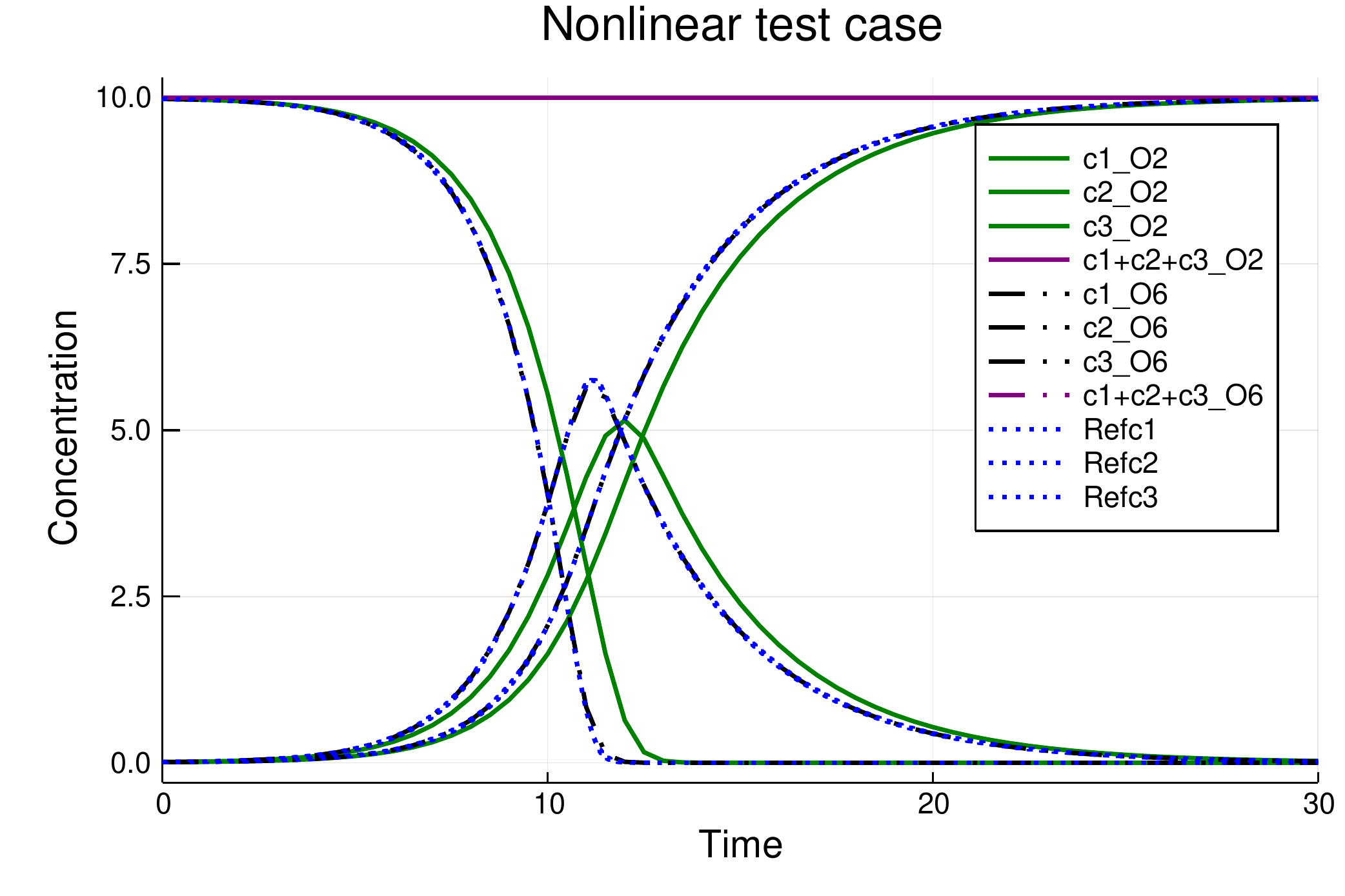}
  \caption{ Second order and sixth order methods together with the reference solution (SSPRK104)
 }
  \label{fig:Non_linear}
\end{figure}

Since we lack of an analytical solution, in the error plots, we compare successive errors between two refinements of the time mesh 
\begin{equation}
\E_N=\frac{1}{N} \sum_{n=1}^N \left( \frac{1}{I} \sum_{i=1}^I \left(c_{i,N}^n-c_{i,2N}^{2n} \right)^2 \right)^\frac{1}{2}. \end{equation}
Here, the subscript $N$ indicates the number of equispaced timesteps used to subdivide the total time interval.
The results are presented in Figure \ref{fig:Non_Linear_model_error}. As for the linear case, we can see that the error decay fulfils the expected behavior and that the order of accuracy tends to the correct one.
The slight decrease of the slope function in the right picture using sixth order can be explained by the fact that the error values are close to machine precision in that area and this causes the deprecation of the slope. \\

These plots verify our theoretical investigations from section \ref{sec:mPDeC}.

\begin{figure}[!htp]
\centering
    \includegraphics[width=0.48\textwidth]{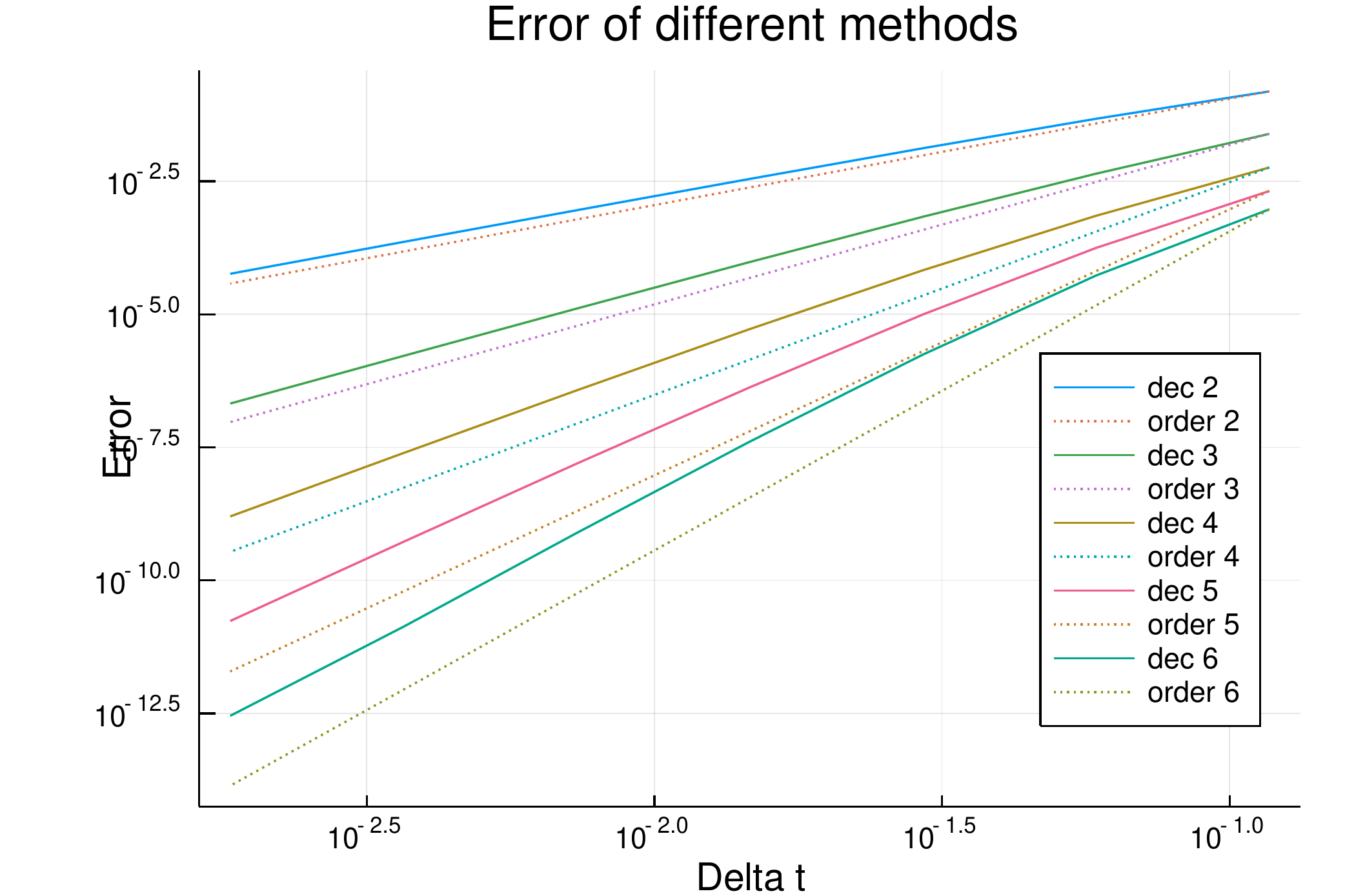}
    \includegraphics[width=0.48\textwidth]{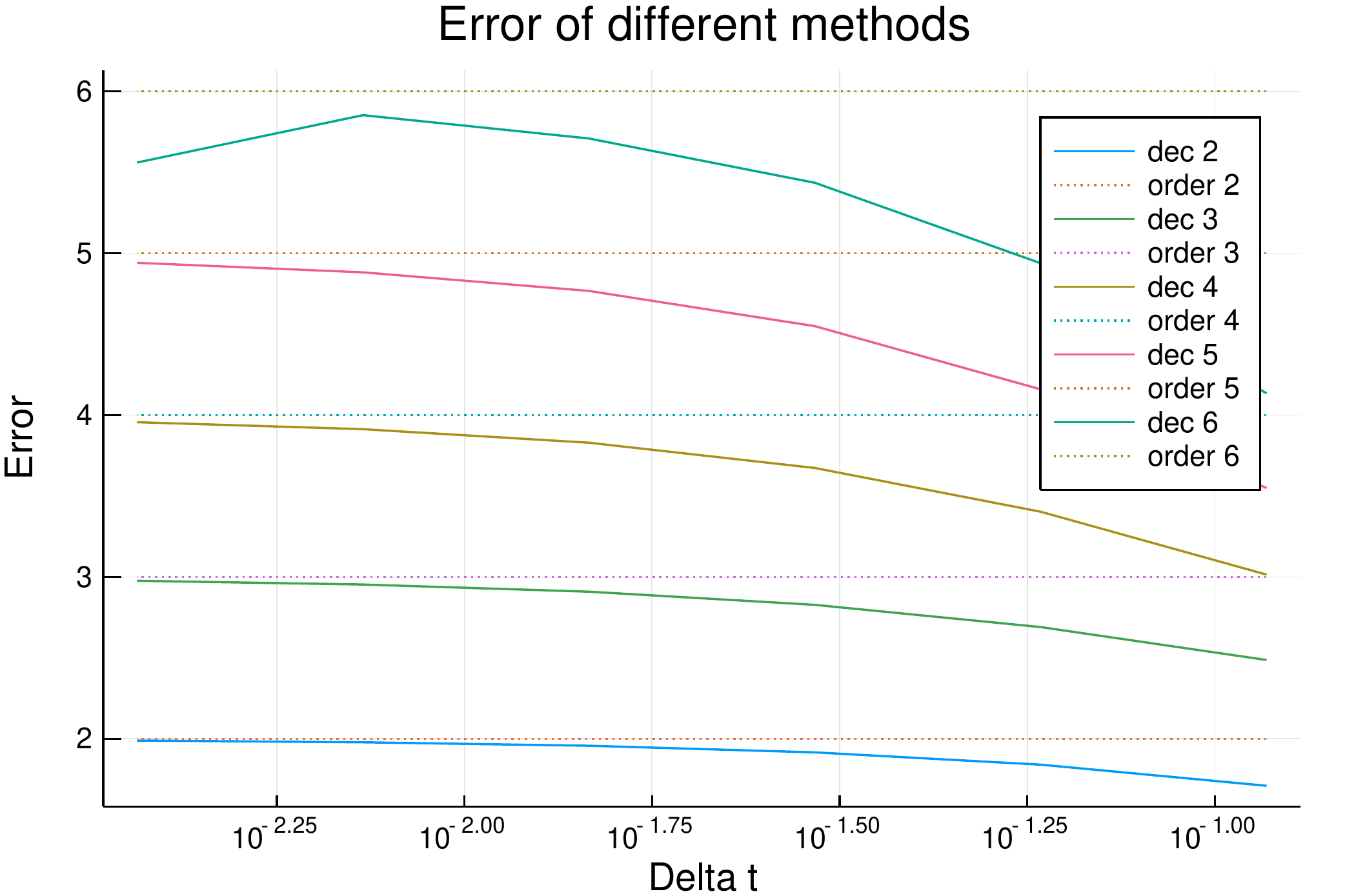}
  \caption{ Second to sixth order error behaviors and slopes of the errors
 }
  \label{fig:Non_Linear_model_error}
\end{figure}

\subsection{Robertson Test case }
In the last test case, we prove the robustness of the mPDeC schemes in presence of stiff problems. The proposed test is the Robertson problem for a chemical reaction system.
It consists of  
\begin{equation}\label{eq:Robertson}
 \begin{aligned}
  c_1'(t)&=10^4c_2(t)c_3(t)-0.04c_1(t)\\
  c_2'(t)&= 0.04c_1(t)-10^4c_2(t)c_3(t)-3\cdot 10^7c_2(t)^2\\
  c_3'(t)&=3\cdot 10^7c_2(t)^2
 \end{aligned}
\end{equation}
with initial conditions
$\bc^0=(1,0,0)$.\footnote{To avoid the division by zero in the mPDeC scheme,
we slightly modify the initial condition in the practical implementation, i.e.,  $\bc^0=(1-2eps, eps, eps)$ with $eps= 2.22 \cdot  10^{-16}$.}
The time interval of interest is $[10^{-6}, 10^{10}]$.
The PDS for \eqref{eq:Robertson} reads
\begin{equation*}
  p_{1,2}(\bc)=d_{2,1}(\bc)=10^4c_2(t)c_3(t), \quad  p_{2,1}(\bc)=d_{1,2}(\bc)=0.04c_1(t),\quad   
  p_{3,2}(\bc)=d_{2,3}(\bc)=3\cdot 10^7c_2(t)
\end{equation*}
and zero for the other combinations.\\
In the Robertson test case, the numerical scheme has to deal with several time scales. Therefore, a constant time step size is not suitable for this purpose. 
\begin{figure}[!htp]
\center
    \includegraphics[width=0.8\textwidth]{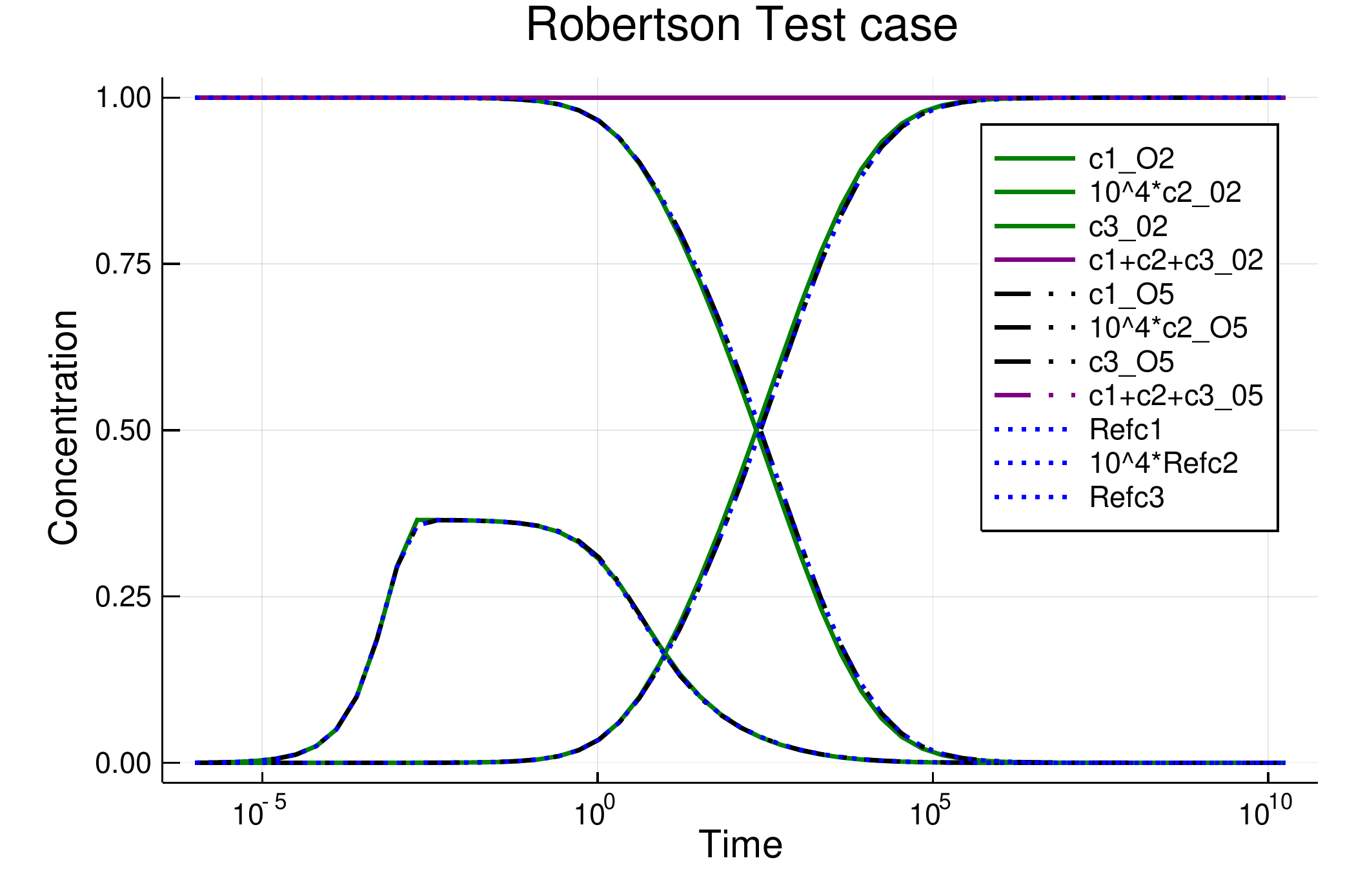}
  \caption{ Second and fifth order solutions and references 
 }
  \label{fig:Robertson}
\end{figure}
Following again the literature \cite{kopecz2018order}, we use increasing time steps $\Delta t_n=2^{n-1}\Delta t_0$ with $\Delta t_0=10^{-6}$,  where $n$ indicates the $n$-th timestep. To make the small $c_2$ values visible on the plot, we 
multiply it by $10^4$. As a comparison, we calculate the reference solution (dotted, blue line) using the 
function  Rodas4\footnote{ A 4-th order A-stable stiffly stable Rosenbrock method with a stiff-aware 3rd order
interpolant.} from Julia, where we split 
the time-interval into 55 subdomains and we solve it on every subdomain with relative tolerance $10^{-20}$ and absolute tolerance $10^{-20}$.
We plot again a second order (green, solid lines) and fifth order (black, dashed-dotted lines) approximations generated by the mPDeC methods and, as it can be seen in figure \ref{fig:Robertson}, the designed methods produce reliable and robust results for this kind of stiff problems. As always, the conservation and the positivity properties are fulfilled.

Finally, we can say that the simulations run in this section express the quality of the mPDeC schemes.
Moreover, they show that all the targeted properties are obtained even for very problematic test cases.

\section{Summary and Outlook}

In this paper, we presented a way to build positivity preserving, conservative and arbitrary high-order numerical schemes
for production-destructions systems of equations. We adapted the idea of \cite{burchard2003high} to build modified 
Patankar type schemes to the Deferred Correction method as an underlying 
scheme. By altering  the $\L^1$ and $\L^2$ operators using the modified Patankar trick, we were able 
to obtain schemes with the desired properties. We proved that the proposed modified Patankar DeC (mPDeC) schemes are
arbitrary high-order, conservative and positivity preserving. In numerical simulations,
we confirmed our theoretical 
considerations with various test cases.
\\
However, further research can be pursued in this direction. As it was 
investigated in  \cite{kopecz2018unconditionally,kopecz2018order} for families of MPRK, it is possible
to study the accuracy and the stability of the method varying the weightings of the production-destruction
terms of the schemes. In the spirit of the work \cite{kopecz2018order},
a change of the weighting of the Patankar modification in the $\L^1$ and $\L^2$ operators 
should be easily applicable to the mPDeC schemes and theoretical investigations will be 
considered in future research, in particular regarding the stability conditions. 
Also the distribution of the subtimesteps between $t^n$ and $t^{n+1}$ plays a big role
on stability and accuracy of the scheme. Many choices are valid and the possible influence
of the properties of the method must be carefully analysed. This idea is already work in progress
for the classical DeC approach and will be extended to the mPDeC version in the future. 
Finally, we want to apply and analyse this type of schemes in context of partial differential equations.
Here, we focus on applications and problems as described in  \cite{huang2019positivity,huang2018third,meister2014unconditionally}.
As one can see, there are still many open questions and tasks for the mPDeC schemes and we are looking 
forward to continue our work in this field.

\appendix
\section{Notation}\label{sec:notation}
We provide a small table \ref{tab:notation} with the notation of symbols used along the paper. Even if some of the notations are ambiguous, the used indices should always clarify the referred meaning. We prefer to keep this notation to keep fluid the reading. 
\begin{table}
\begin{tabular}{r|l}
Notation & Meaning\\
\hline
$I$ & Number of constituents and dimension of the ODE system\\
$i$ & Index for constituents\\ 
$c_i$ & Value of the $i$th constituents\\
$\bc$ & Vector of all the constituents $\bc=(c_1, \dots, c_I )^T$\\
$N$ & Number of time intervals\\
$n$ & Index for a timestep\\
$t^n$ & Timestep\\
$\bc^n$ & Variables at timestep $t^n$\\
$M$ & Number of subtimeintervals in a timeinterval\\
$m$ & Index for subtimesteps\\
$t^{n,m}=t^m$ & Subtimestep\\
$\bc^{n,m} = \bc^m$ & Variable at subtimestep $m$\\
$\bbc$ & Vector of variables at all subtimestep $m=0,\dots, M$\\
$K$ & Number of iterations of the DeC procedure\\
$(k)$ & Index of the iteration\\
$\bc^{n, m,(k)}=\bc^{m,(k)}$ & Variables for timestep $n$ at the subtimestep $m$ and iteration $k$\\
$\bbc^{(k)}$ & Vector of variables for all subtimesteps $m=0,\dots, M$ at the iteration $k$\\
$\L^1 (\cdot) $ & First order operator of DeC procedure\\
$\L^2 (\cdot) $ & High order operator of DeC procedure\\
$\L^2 (\cdot,\cdot) $ & High order operator of mPDeC procedure\\
$\bbc^*$ & Solution of the system $\L^2(\bbc^*)=0$.

\end{tabular}
\caption{Notation table}\label{tab:notation}
\end{table}

\section{Algorithm}\label{sec:algorithm}
We present a pseudo-code for the creation of the mass matrix in Algorithm \ref{algo:mass} and one for the mPDeC algorithm in box Algorithm \ref{algo:mPDeC}. Both algorithms are very simple. The first one consists of 3 loops: 2 for the constituents $i,j=1,\dots, I$ and one for the subtimesteps $r=0,\dots, M$ and an if statement. The second one consists of 3 nested loops: one for timesteps $\lbrace t^n \rbrace_{n=0}^N$, one for corrections of the DeC algorithm $k=1,\dots, K$ and one for the subtimesteps $m=1,\dots, M$.
\begin{algorithm}
	\fontsize{10pt}{10pt}\selectfont
	\caption{Mass} 
	\begin{algorithmic}[1]
		{\REQUIRE Production-destruction functions $p_{i,j}(\cdot),\,d_{i,j}(\cdot)$, previous correction variables $\bbc^{(k-1)}$, actual subtimestep $m$.
		\STATE $\M:=0$
		\FOR{$i=1 $ \TO $ I$}
			\FOR{$j=1$ \TO $I$}
				\FOR{$r=0$ \TO $M$}
					\IF{$\theta_r^m\geq 0$}
						\STATE $\M_{i,j} = \M_{i,j} -\Delta t \theta_r^m \frac{p_{i,j}(\bc^{r,(k-1)})}{c_j^{m,(k-1)}}$
						\STATE $\M_{i,i} = \M_{i,i} +\Delta t \theta_r^m \frac{d_{i,j}(\bc^{r,(k-1)})}{c_i^{m,(k-1)}}$
					\ELSE
						\STATE $\M_{i,j} = \M_{i,j} +\Delta t \theta_r^m \frac{d_{i,j}(\bc^{r,(k-1)})}{c_j^{m,(k-1)}}$
						\STATE $\M_{i,i} = \M_{i,i} -\Delta t \theta_r^m \frac{p_{i,j}(\bc^{r,(k-1)})}{c_i^{m,(k-1)}}$	
					\ENDIF
				\ENDFOR
			\ENDFOR
		\ENDFOR
		}
	\end{algorithmic}\label{algo:mass}
\end{algorithm}

\begin{algorithm}
	\fontsize{10pt}{10pt}\selectfont
	\caption{mPDeC} 
	\begin{algorithmic}[1]
		{\REQUIRE Production-destruction functions $p_{i,j}(\cdot),\,d_{i,j}(\cdot)$, timesteps $\lbrace t^n \rbrace_{n=0}^N$, initial condition $\bc^0$.
			\FOR{$n=1 $ \TO $ N$}
				\FOR{$k=0$ \TO $K$}
					\STATE Set $\bc^{0,(k)}:=\bc^n $
				\ENDFOR
				\FOR{$m=1$ \TO $M$}
					\STATE Set $\bc^{m,(0)}:=\bc^n $
				\ENDFOR
				\FOR{$k=1$ \TO $K$}
					\FOR{$m=1$ \TO $M$}
						\STATE Compute the mass matrix $\M (\bc^{m,(k-1)}):=$Mass$(\bbc^{(k-1)}, m)$ using algorithm \ref{algo:mass}
						\STATE Compute $\bc^{m,(k)}$ solving the linear system $\M(\bc^{m,(k-1)}) \bc^{m,(k)} = \bc^{n}$  given by \eqref{eq:explicit_dec_correction}
					\ENDFOR
				\ENDFOR
				\STATE Set $\bc^{n+1}:=\bc^{M,(K)}$
			\ENDFOR
		}
	\end{algorithmic}\label{algo:mPDeC}
\end{algorithm}

\section*{Acknowledgements}
P. \"Offner has been funded by the the SNF project (Number 175784). \\
Davide Torlo is supported by ITN ModCompShock project funded by the European Union’s Horizon 2020 research and innovation program under the Marie Sklodowska-Curie grant agreement No 642768.

\bibliographystyle{abbrv}
\bibliography{literature}

\end{document}